\documentclass[a4paper,11pt]{article}

\usepackage{amsmath}
\usepackage{amssymb}
\usepackage{amsthm}
\usepackage{amsfonts}
\usepackage{mathtools}
\usepackage{graphicx}
\usepackage{tikz,pgfplots}
\usepackage[utf8]{inputenc}
\usepackage{pgfplots}
\usepackage{tikz}
\pgfplotsset{compat=1.9}
\usepackage{subcaption}
\usepackage{caption}

\usepackage{titlesec}
\usepackage{enumitem}
\usepackage{mathtools}
\usepackage{listings}
\usepackage{geometry}
\usepackage{bbm}
\usepackage{xcolor}
\usepackage{hyperref}
\usepackage{caption}
\usepackage{mathrsfs}
\usepackage{listings}
\usepackage{booktabs}
\usepackage{floatrow}
\newfloatcommand{capbtabbox}{table}[][\FBwidth]
\usepackage{blindtext}
\usepackage{colortbl}
\newtheorem{theorem}{Theorem}[section]
\newtheorem{corollary}[theorem]{Corollary}

\newtheorem{lemma}[theorem]{Lemma}
\newtheorem{remark}[theorem]{Remark}
\newtheorem{definition}[theorem]{Definition}
\newtheorem{conclusion}[theorem]{Conclusion}

\makeatletter
\newcommand*{\rom}[1]{\expandafter\@slowromancap\romannumeral #1@}
\makeatother

\newcommand{\eps}{\varepsilon}

\newcommand{\maxs}{\mbox{\normalfont\tiny max}}
\newcommand{\mins}{\mbox{\normalfont\tiny min}}
\newcommand{\pr}{\mbox{\normalfont\tiny prel}}

\newcommand{\D}{\mathcal{D}}

\newcommand{\R}{\mathbb{R}}
\newcommand{\C}{\mathbb{C}}

\renewcommand{\Re}{\mbox{\rm Re}}

\definecolor{mycolor1}{rgb}{0.00000,0.44700,0.74100}%

\newcommand{\quotes}[1]{``#1''}

\newcommand{\dx}{\hspace{2pt}\mbox{d}x}

\newcommand{\ci}{\mathrm{i}} 

\newcommand{\uast}{u^{\ast}}

\newcommand{\boldS}{\pmb{\mathbb{S}}}
\newcommand{\bfu}{{\bf u}}
\newcommand{\bfv}{{\bf v}}
\newcommand{\bfw}{{\bf w}}
\newcommand{\bfd}{{\bf d}}
\newcommand{\bfL}{{\mathbf L}}
\newcommand{\bfH}{{\mathbf H}}
\newcommand{\bfzero}{\boldsymbol{0}}

\newcommand{\conju}[1]{{\ensuremath{\overline{u}_{#1}}}}
\newcommand{\conjw}[1]{{\ensuremath{\overline{w}_{#1}}}} 
\newcommand{\conjv}[1]{{\ensuremath{\overline{v}_{#1}}}}
\newcommand{\Hone}{{\ensuremath{\boldsymbol{H}^1(\D)}}}
\newcommand{\Ltwo}{{\ensuremath{\boldsymbol{L}^2(\D)}}}
\newcommand{\uprel}[2]{\ensuremath{u_{\pr,#1}^{#2}}}
\newcommand{\bfuprel}[1]{\ensuremath{\boldsymbol{u}_{\pr}^{#1}}}

\newcommand{\energynorm}[1]{\|\hspace{-1pt}|#1\|\hspace{-1pt}|} 

\begin{document}

\begin{center}
{\LARGE 
Metric-driven numerical methods
\renewcommand{\thefootnote}{\fnsymbol{footnote}}\setcounter{footnote}{0}
\hspace{-3pt}\footnote{The authors acknowledge the support by the Deutsche Forschungsgemeinschaft (DFG, German Research Foundation) where P. Henning and L. Huynh are supported through the project grant 551527112 and D. Peterseim  through the project grant 564828373.}}\\[2em]
\end{center}

\begin{center}
{\large Patrick Henning,\hspace{-2pt}\footnote[1]{\label{affiliation}Department of Mathematics, Ruhr-University Bochum, DE-44801 Bochum, Germany.\\ Email: \href{mailto:patrick.henning@rub.de}{patrick.henning@rub.de}, \href{mailto:Laura.Huynh@rub.de}{laura.huynh@rub.de}} \,\,
Laura Huynh\textsuperscript{\ref{affiliation}} 
\,\,and\,\, Daniel Peterseim
\footnote[2]{\label{affiliation-2}Institut für Mathematik \& Centre for Advanced Analytics and Predictive Sciences (CAAPS), University Augsburg, DE-86159 Augsburg.\,\, Email: \href{mailto:daniel.peterseim@uni-a.de}{daniel.peterseim@uni-a.de}}
}\\[2em]
\end{center}

\abstract{In this work, we explore the concept of {\it metric-driven numerical methods} as a powerful tool for solving various types of multiscale partial differential equations. Our focus is on computing constrained minimizers of functionals - or, equivalently, by considering the associated Euler-Lagrange equations - the solution of a class of eigenvalue problems that may involve nonlinearities in the eigenfunctions. 
We introduce metric-driven methods for such problems via Riemannian gradient techniques, leveraging the idea that gradients can be represented in different metrics (so-called {\it Sobolev gradients}) to accelerate convergence. We show that the choice of metric not only leads to specific metric-driven iterative schemes, but also induces approximation spaces with enhanced properties, particularly in low-regularity regimes or when the solution exhibits heterogeneous multiscale features.
In fact, we recover a well-known class of multiscale spaces based on the {\it Localized Orthogonal Decomposition} (LOD), now derived from a new perspective. Alongside a discussion of the metric-driven approach for a model problem, we also demonstrate its application to simulating the ground states of spin-orbit-coupled Bose-Einstein condensates.}

\section{Introduction}
Solving partial differential equations (PDEs) with a complex underlying structure, such as multiscale variations or nonlinear features, can be challenging. Standard numerical approaches often require extremely fine computational meshes to resolve all structures, hence resulting in high-dimensional approximation spaces. At the same time, iterative solvers for handling the nonlinearities can suffer from very slow convergence. In such cases it is possible to take the structure of the PDE into account to construct problem-specific numerical methods. 

For example, classical approaches for multiscale PDEs are based on enriching the shape functions of standard approximation spaces by multiscale features. This can be achieved by solving local problems based on the differential operator. The corresponding local solutions naturally contain structural information about the PDE and can then be used to correct/enrich the original shape functions. The resulting approximation spaces are called multiscale spaces, they are typically of low dimension and have very good approximation properties with respect to the solution of the underlying PDE. There is a vast landscape of different approaches how this can be realized, where we exemplarily mention the heterogeneous multiscale method (HMM) \cite{AEE12,EE03}, the multiscale finite element method (MsFEM) \cite{HW97,EH09}, the generalized MsFEM (GMsFEM) \cite{ChEfHo23,EfGaHo13}, the multiscale spectral generalized finite element method (MS-GFEM) \cite{BaLi11,BaLiSiSt20,MaSc22,MaScDo22}, gamblet-based methods \cite{Owhadi17,OwhadiScovel19},
Variational Multiscale Methods (VMM) \cite{HFM98,HuS07,LaM07,Shrestha-1}
or the localized orthogonal decomposition (LOD) \cite{ActaLOD21,MaP14,LODbook21}.  
In this contribution, we shall especially focus on the LOD methodology where we will elaborate on why this class of multiscale methods can be seen as a metric-driven approach, obtained by an \quotes{energy metric} which is induced by the differential operator.

Besides the construction of approximation spaces, problem-specific information can also be used for the design of iterative methods. A natural application is the use of gradient methods to find minimizers of functionals $E$. Gradient methods are based on taking a step, from a current location $u^0$, into the direction of the (classical) steepest descent $-E^{\prime}(u^0)$ to get closer to the unknown minimizer. However, the notion of the steepest descent heavily depends on the metric in which we measure \quotes{steepness}. Even though this is often not explicitly mentioned, classical gradient methods are based on the canonical $L^2$-metric. On the contrary, alternative metrics can accelerate the gradient descent in various applications, cf. \cite{AltPetSty22,CLLZ24,DaK10,DaP17,HenJar24,WuLiuCai2025,ZhangQang2024}, where the representation of gradients in different metrics is formalized by the concept of Sobolev gradients \cite{Neu97}. By choosing the metric adaptively depending on the operator $E^{\prime}(v)$ (or linearized versions of it), additional problem-specific information can be incorporated to select an improved descent direction. This strategy of a metric-driven gradient descent was first proposed in \cite{HeP20}. 

So far, both paths of metric-driven numerical methods (i.e. the construction of problem-dependent approximation spaces and iterative solvers) were treated independently in the literature as essentially two disjoint fields. In this work we want to establish new links between the seemingly disconnected paths. Starting from a (constrained) minimization problem, we discuss how a particular energy-metric in the spirit of \cite{HeP20} gives simultaneously rise to a metric-driven gradient method and a corresponding metric-driven approximation space. As for the approximation space, we recover the aforementioned LOD construction, but derived from a new perspective. We first describe the construction for a simple quadratic minimization problem and then discuss its generalization to a more complicated minimization problem of Gross-Pitaevskii type. Finally, to show that the ideas are not restricted to model problems, we also present an application in quantum physics where we compute the ground states of so-called pseudo-spin-1/2 Bose--Einstein condensates.

Before sketching the outline of this paper, it is important to distinguish our notion of metric-driven methods from the concept of \quotes{metric-based upscaling} \cite{OwZ07} for elliptic PDEs $\mathcal{A}u := -\nabla \cdot (A \nabla u) =f$. The latter is based on representing the PDE in an $\mathcal{A}$-harmonic coordinate system in which heterogeneous solutions become smooth. This is not related to our approach.

$\\$
{\it Outline.} In Section \ref{section:2} we introduce the concept of metric-driven discretizations for a quadratic minimization problem with an $L^2$-normalization constraint. The generalization to minimization problems with higher order polynomial contributions is given in Section \ref{section:metric-adaptive-GPE}. Finally, in Section \ref{section:metric-adaptive-SOBEC} we present the application to spin-orbit coupled Bose--Einstein condensates where we formulate and analyze a corresponding metric-driven iterative method which is afterwards illustrated in corresponding numerical experiments.

\section{The concept of metric-driven discretizations}\label{section:2}

To introduce the concept of metric-driven discretizations, we consider a simple model problem that seeks the smallest eigenvalue and a corresponding eigenfunction of a linear elliptic differential operator, potentially involving rapidly varying multiscale coefficients. This simplified setting allows us to illustrate the basic ideas of metric-driven schemes before turning to more involved nonlinear problems later in the chapter.

\subsection{A model problem}

In this section we will briefly introduce the setting of the model problem and why it is suitable in this context.
\subsubsection{Motivation}

The model problem considered in this section serves as a simple setting in which the main ideas of metric-driven schemes can be illustrated transparently. 
Many PDE problems of practical interest admit a variational formulation in which a suitable energy functional is minimized under an $L^2$-normalization constraint. This
is the case, for instance, in the computation of principal eigenpairs of elliptic operators arising in models of diffusion, wave propagation, and other physical or engineering applications.

From an algorithmic point of view, such problems are typically solved by iterative methods. Two fundamental design choices then arise: first, how the infinite-dimensional problem is discretized in space, and second, how the resulting minimization problem is solved iteratively. The key observation underlying metric-driven methods is that both aspects are strongly influenced by the choice of the metric in which gradients are represented. While the minimizer itself is independent of this choice, the metric determines the descent direction and therefore
the geometry and efficiency of the iteration. In the following subsections we illustrate this idea for a simple elliptic model problem.
\subsubsection{Analytical setting}

In the following subsection we introduce the analytical setting of the model problem. For that, let $\mathcal{D} \subset \mathbb{R}^d$ ($d \in \{1,2,3\}$) denote a bounded Lipschitz domain with polygonal boundary and let
\begin{align*}
A \in L^{\infty}(\D, \mathbb{R}^{d \times d})
\qquad
\mbox{and}
\qquad
V\in L^{\infty}(\D,\mathbb{R}_{\ge0})
\end{align*}
be two given coefficient functions, where the matrix-valued coefficient $A(x)$ is assumed to be symmetric for almost all $x\in \Omega$ and uniformly elliptic, i.e., $\sup\limits_{\xi \in \R^d\setminus \{0\}} \frac{A \xi \cdot \xi}{|\xi|^2} \ge \alpha >0$ for some positive constant $\alpha$.

With this, we consider the functional $E:H^1_0(\D) \rightarrow \R_{\ge0}$ given by
\begin{align*}
E(u ) := \frac{1}{2} \int_{\D} |A^{1/2} \nabla u|^2 + V |u|^2  \,\,\mbox{d}x
\end{align*}
and we seek a minimizer $\uast$ of $E$ on $H^1_0(\D)$ under the normalization constraint $\| \uast \|_{L^2(\D)}=1$. For brevity, we define the corresponding constraint manifold ($L^2$-sphere in $H^1_0(\D)$) by
\begin{align*}
\mathbb{S} \,\,:=\,\, \{ v \in H^1_0(\D) \,\, | \,\, \| v \|_{L^2(\D)}=1 \}
\end{align*}  
such that we can write the relevant minimizers compactly as
\begin{align}
\label{energy-minimization-problem-linear-setting}
\uast \,=\, \underset{ v \in \mathbb{S}}{\mbox{arg\hspace{2pt}min}} \, E(v).
\end{align}
It is well-known (cf. \cite{CCM10}) that the minimizer exists, that it is unique up to sign and that it is either strictly positive or strictly negative in the interior $\D$. Problems of this type arise in many applications involving elliptic operators, including diffusion and conductivity problems, vibration analysis, and various models in physics and engineering.

The problem \eqref{energy-minimization-problem-linear-setting} can be equivalently expressed as an eigenvalue problem. In fact, by introducing a Lagrange multiplier $\lambda \in \mathbb{R}$ for the constraint $\uast \in \mathbb{S}$, the corresponding Euler--Lagrange equations $E^{\prime}(\uast) = \lambda \,( \uast ,\cdot )_{L^2(\D)}$  take the form of a linear eigenvalue problem. Computing the Fr\'echet derivative $E^{\prime}(u)$ in a straightforward manner, the resulting eigenvalue problem hence seeks an $L^2$-normalized eigenfunction $\uast \in H^1_0(\D)$ and corresponding eigenvalue $\lambda \in \mathbb{R}$ such that
\begin{align}
\label{eigenvalue-problem-linear-setting}
( A \nabla \uast , \nabla v )_{L^2(\D)} + ( V \uast , v)_{L^2(\D)} \, =\, \lambda \, (\uast, v)_{L^2(\D)}
\qquad \mbox{for all } v\in H^1_0(\D).
\end{align}
In strong form, the eigenvalue problem can be written as
\begin{align*}
\mathcal{L} \uast \, := \, - \nabla \cdot (A \nabla \uast) \,+ \, V \uast \,\, = \,\, \lambda \, \uast,
\end{align*} 
as an identity of $L^2$-functions. Since $\mathcal{L}$ represents a linear elliptic differential operator, it is well-known that an eigenfunction to the {\it smallest eigenvalue} $\lambda$ (of $E^{\prime}(\uast) = \lambda \,( \uast ,\cdot )_{L^2(\D)}$, or equivalently, $\mathcal{L} \uast = \lambda \uast$) is precisely a minimizer to \eqref{energy-minimization-problem-linear-setting}. Furthermore, the eigenvalue $\lambda$ is positive and simple.

Finding a minimizer of $E$, or respectively the smallest eigenvalue of the linear operator $E^{\prime}$, is a classical problem that can be solved with various standard techniques. We will not discuss these standard techniques here since the whole purpose of this section is an introduction to the concept of metric-driven schemes in a simple setting. For that, we have to distinguish two aspects of the discretization of \eqref{eigenvalue-problem-linear-setting}: an iterative solver for computing a relevant eigenpair and a suitable spatial discretization of each iteration step. These two aspects are addressed individually in the next two subsections.

\subsection{Metric-driven steepest descents and Riemannian Sobolev gradients}
\label{subsection-metric-driven-grad-flow-linear}
Recalling that problems \eqref{energy-minimization-problem-linear-setting} and \eqref{eigenvalue-problem-linear-setting} are equivalent for the smallest eigenvalue, we can take any of the two perspectives. In this subsection, the energy minimization perspective will be deployed. 

To find a constrained minimizer of $E$ in the sense of \eqref{energy-minimization-problem-linear-setting}, it is possible to consider a corresponding {\it Riemannian} gradient flow on the (Hilbert) manifold $\mathbb{S}$. Such a gradient flow follows the direction of the Riemannian steepest descent on $\mathbb{S}$. Usually, steepest descent is measured with respect to the $L^2$-metric. However, alternative metrics can lead to significantly faster energy decay, which can be exploited in the construction of efficient numerical methods.

Before introducing Sobolev gradients (which depend on the chosen metric), let us briefly clarify the notation for derivatives of functionals. For a functional
$E : H_0^1(\D) \to \R$, we write $E'(u)$ for its Fr\'echet derivative at $u$. Thus $E'(u)$ is a continuous linear functional acting on perturbations
$v \in H_0^1(\D)$, and the expression $\langle E'(u), v \rangle$ denotes the first variation of $E$ at $u$ in direction $v$, i.e., the canonical duality pairing between the functional $E'(u)$ and the perturbation $v$. 

In classical optimization in $\R^n$, the derivative of a function $E$ at a point $x$ can be identified with a vector, namely the gradient $\nabla E(x)$, which points in the direction of steepest ascent. Consequently, $-\nabla E(x)$ points in the direction of steepest descent. In infinite-dimensional settings this identification is no longer automatic: the derivative $E'(u)$ is a linear functional and therefore does not directly define a descent direction. To obtain a direction of steepest descent, one has to represent this functional by an element of a Hilbert space using a suitable inner product. The choice of this inner
product determines how steepness of the energy landscape is measured and therefore influences the resulting descent direction. We will make this precise in the following subsection.

\subsubsection{Sobolev gradients}
The role of the metric can be understood through Sobolev gradients $\nabla_X$ of functionals $E$, cf. \cite{Neu97}. Loosely speaking, Sobolev gradients allow to characterize what we mean by steepest descent (or steepest ascent) on the graph of $E$ when \quotes{steepness} is measured in different metrics. This is a relevant practical aspect of gradient descent methods since the choice of the metric can have a major influence on how fast the gradient descent approaches a minimizer. In particular, the metric can be viewed as a built-in preconditioner: it does not change the minimizer itself, but it may significantly change the geometry of the descent path and hence the efficiency of the iterative method.

Let us now make the definition of Sobolev gradients more precise. Given a Hilbert space $X$ with some inner product $(\cdot,\cdot)_X$ (which represents a selected metric) the $X$-Sobolev gradient $\nabla_X E(u)$ of the functional $E$ in some point $u \in H^1_0(\D)$ is given by the Riesz-representation of the \quotes{regular gradient} $E^{\prime}(u)$ in the space $X$, that is $\nabla_X E(u) \in X$ solves
\begin{align}
\label{def-Sobolev-gradient}
( \nabla_X E(u) , v)_{X} = \langle E^{\prime}(u) , v \rangle \qquad \mbox{for all }  v \in X.
\end{align} 
 At first glance, the above definition only makes sense if $X \subset H^1_0(\D)$ such that $\langle E^{\prime}(u) , v \rangle$ is well defined for $v \in X$. However, note that this can be compensated by smoothness of $u$ and/or $A$. For example, if $u\in H^1_0(\D)$ is such that $\nabla \cdot (A \nabla u) \in L^2(\D)$, then $X=L^2(\D)$ is admissible and we obtain (by integration by parts) the identification
\begin{align*}
\nabla_{L^2} E(u) \,=\, - \nabla \cdot (A \nabla u) + V u \,\,\in\,\, L^2(\D).
\end{align*}
There are various choices for $X$, where we refer to \cite{DaK10,DaP17,HenJar24,KaE10} for corresponding examples and numerical studies. A particular choice was suggested in \cite{HeP20}, motivated by the idea to select the metric in an optimal way such that the Sobolev gradient becomes the identity, i.e., such that $\nabla_X E(u)= u$. For general functionals $E$, there is no fixed metric $(\cdot,\cdot)_X$ that fulfills this property for arbitrary points $u \in \mathbb{S}$, which is why such a metric needs to adaptively change depending on the location on the manifold $\mathbb{S}$. We will elaborate more on this aspect in Section \ref{section:metric-adaptive-GPE}. However, in the simplified setting of this section, $E^{\prime}$ is a linear operator and a suitable universal metric on $H^1_0(\D)$ is canonically given by the inner-product
\begin{align}
\label{optimal-X-metric-linear}
(u, v)_X \,\,:=\,\, \langle E^{\prime}(u) , v \rangle
\,\, = \,\, ( A \nabla u , \nabla v )_{L^2(\D)} + ( V\, u , v )_{L^2(\D)}.
\end{align}
Indeed, by definition we have $(\nabla_X E (u), v)_X = \langle E^{\prime}(u) , v \rangle= (u, v)_X$ for all $v \in H^1_0(\D)$ and hence $\nabla_X E (u) = u $ for all $u\in H^1_0(\D)$. As we will exploit the above metric repeatedly, we introduce the corresponding solution operator $\mathcal{L}^{-1} : H^1_0(\D) \rightarrow H^1_0(\D)$ as $\mathcal{L}^{-1} u \in H^1_0(\D)$  solving
\begin{align}
\label{def-Linv}
( \mathcal{L}^{-1} u , v)_X  = ( u , v)_{L^2(\D)} \qquad \mbox{for all } v \in H^1_0(\D),
\end{align}
which effectively gives us $E^{\prime}(\mathcal{L}^{-1} u) = u = \nabla_X E (u)$.

\subsubsection{Riemannian Sobolev gradients}
Since we are not interested in the general steepest descent of $E$ in $u \in \mathbb{S}$, but only in the steepest descent among all admissible directions that keep us on the manifold $\mathbb{S}$, we need to consider Riemannian Sobolev gradients, i.e., the projection of Sobolev gradients onto the tangent space of the manifold given by 
$$
T_u \mathbb{S} \,= \, \{  v \in H^1_0(\D) \, | \, (u,v)_{L^2(\D)} = 0  \}.
$$
While the metric-dependency already enters through the Sobolev gradient, it also determines how this gradient is projected onto the tangent space via the $X$-orthogonal projection $P_{u,X} : X \to T_u \mathbb{S}$, that is, the projection $P_{u,X}(v) \in T_u \mathbb{S}$ that fulfills
\begin{align}
\label{orth-proj-tangent-space}
( P_{u,X} (v) , w )_X = ( v, w)_X \qquad \mbox{for all } w\in T_u \mathbb{S}.
\end{align}
This projection ensures that the descent direction remains feasible, i.e., it preserves the constraint $u \in \mathbb{S}$ while retaining the steepest descent property with respect to the $X$-metric. With this, the Riemannian $X$-Sobolev gradient of $E$ in $u\in \mathbb{S}$ is given by $(P_{u,X} \circ \nabla_X E)u$ and, consequently, $-(P_{u,X} \circ \nabla_X E)u$ defines the Riemannian steepest descent of $E$ in $u$ with respect to the $X$-metric.
 
We can now apply these considerations to the particular choice of $X$ defined in \eqref{optimal-X-metric-linear}. For $u \in \mathbb{S}$, the projection $P_{u,X}$ is then explicitly given by
 \begin{align*}
 P_{u,X}(v) = v - \frac{(u,v)_{L^2(\D)}  }{ ( \mathcal{L}^{-1} u, u )_{L^2(\D)} }   \mathcal{L}^{-1} u.
 \end{align*}
In fact, $ ( P_{u,X}(v) , u)_{L^2(\D)}=0$ and $( P_{u,X}(v) , w )_X= (v,w)_X $ for all $w \in T_{u} \mathbb{S}$ because of $( \mathcal{L}^{-1} u , w)_X = (u,w)_{L^2(\D)}=0$. Furthermore, recalling $\nabla_X E(u)=u$, we obtain the corresponding Riemannian $X$-Sobolev gradient as
\begin{align}
\label{Riemannian-Sob-X-metric}
(P_{u,X} \circ \nabla_X E)u \,=\, u \,-\, ( \mathcal{L}^{-1} u, u )_{L^2(\D)}^{-1} \,\mathcal{L}^{-1} u.
\end{align}
\subsubsection{Metric-driven gradient flow}
We are now ready to formulate a Riemannian Sobolev gradient flow for the minimization problem \eqref{energy-minimization-problem-linear-setting}. It is obtained by following the Riemannian steepest descent in the $X$-metric, i.e., it is of the form
\begin{align*}
u^{\prime}(t) = - ( P_{u(t),X} \circ \nabla_X E)u(t) \qquad \mbox{for } t\ge 0
\end{align*}
and some suitable initial value $u(0) =u_0 \in \mathbb{S}$. For the choice $(u,v)_X=\langle E^{\prime}(u) , v \rangle$ as in \eqref{optimal-X-metric-linear} we obtain with the representation \eqref{Riemannian-Sob-X-metric} that the gradient flow $u \in C^1(0,\infty; \mathbb{S})$ fulfills
\begin{align} 
\label{label-gradient-flow-linear-problem}
u^{\prime}(t) = -u(t) \,+\, ( \mathcal{L}^{-1} u(t) , u(t) )_{L^2(\D)}^{-1} \,\mathcal{L}^{-1} u(t).
\end{align}
Well-posedness and energy dissipation $E(u(t)) \le E(u(s))$ for any $s\le t$ can be established as in \cite{HeP20}. Furthermore, for all suitable initial values $u_0 \in \mathbb{S}$ (which typically have to be non-negative), the flow converges exponentially fast to a minimizer $\uast$ of $E$, i.e., for all $\delta>0$ and all sufficiently large times $t\ge t_{\delta}$ 
\begin{align}
\label{continuous-grad-flow-linear-problem-rate}
\| \uast - u(t) \|_{H^1(\D)} \,\, \le \,\, c_{\delta}\, \exp( -(1 - \tfrac{\lambda_1}{\lambda_2} - \delta) t ), 
\end{align}
where $c_{\delta}>0$ is a generic constant that depends on $\delta$ and where $\lambda_1=\lambda>0$ is the smallest and $\lambda_2 >\lambda_1$ the second smallest eigenvalue  of \eqref{eigenvalue-problem-linear-setting}.

\subsubsection{Metric-driven Riemannian gradient method}
Motivated by the fast decay of the gradient flow \eqref{label-gradient-flow-linear-problem} to a minimizer of $E$ on $\mathbb{S}$, it is natural to build numerical methods based on corresponding time discretizations. Note that $t$ is just a pseudo-time and that we are not interested in the time evolution but only in the limiting state for $t\rightarrow \infty$. Consequently, a simple forward Euler discretization of \eqref{label-gradient-flow-linear-problem} is sufficient. For that, we let $u^0 \in \mathbb{S}$ denote a starting value with $u_0 \ge 0$ and $\tau_n>0$ a sequence of (pseudo) time step sizes. For $n\ge 0$, the iterations read
\begin{equation}
\label{RSG-linear}
  \begin{aligned}
u^{n+1}_{\pr} &\,\,:=\,\, (1 - \tau_n)\,u^n  \,+\, \tau_n \, ( \mathcal{L}^{-1} u^n , u^n )_{L^2(\D)}^{-1} \,\mathcal{L}^{-1} u^n,\\
u^{n+1} &\,\,:=\,\, \frac{\hspace{-22pt}u^{n+1}_{\pr} }{\| u^{n+1}_{\pr} \|_{L^2(\D)}},
  \end{aligned}
\end{equation}
where we normalize after each iteration to respect the constraint $u^{n+1} \in \mathbb{S}$ which is not automatically fulfilled by a pure Euler step. Note that the iterations can be equivalently written as
\begin{eqnarray}
\label{Rie-grad-method-linear}
u^{n+1} &=& \frac{ \hspace{-23pt}u^n \, - \, \tau_n ( P_{u^n,X} \circ \nabla_X E)u^n }{\| u^n \, - \, \tau_n ( P_{u^n,X} \circ \nabla_X E)u^n  \|_{L^2(\D)}},
\end{eqnarray}
where $(P_{u^n,X} \circ \nabla_X E)u^n$ denotes the Riemannian Sobolev gradient in the $X$-metric from \eqref{optimal-X-metric-linear}. Hence, the scheme has a natural interpretation as a Riemannian Sobolev gradient method if the step size is chosen such that the energy reduction per iteration is maximized, that is,
\begin{eqnarray*}
\tau_n &:=&
\underset{0<\tau \le 2}{\mbox{arg\hspace{1pt}min}} \,\, E\hspace{-1pt}\left(
\frac{(1 - \tau)\,u^n  \,+\, \tau \, ( \mathcal{L}^{-1} u^n , u^n )_{L^2(\D)}^{-1} \,\mathcal{L}^{-1} u^n}{\| (1 - \tau)\,u^n  \,+\, \tau \, ( \mathcal{L}^{-1} u^n , u^n )_{L^2(\D)}^{-1} \,\mathcal{L}^{-1} u^n \|_{L^2(\D)} }
\right).
\end{eqnarray*} 
The constraint $\tau_n<2$ in the minimization is analytically justified by the observation that the iterations cannot decrease the energy if $\tau_n \ge 2$. On the contrary, it can be proved, that the iterations are guaranteed to reduce the energy for all sufficiently small step sizes $\tau_n$ and that they converge to the unique positive minimizer of problem \eqref{energy-minimization-problem-linear-setting}. We summarize these properties in the following theorem and refer to \cite{HeP20} for a proof.
\begin{theorem}
\label{theorem:global-convergence-linear-setting}
In the general setting of this section, let $u^0 \in \mathbb{S}$ denote a nonnegative starting value and consider the iterations given by \eqref{RSG-linear}. Then there exists a step size interval $[\tau_{\mins}, \tau_{\maxs}] \subset (0,2)$ such that for all $\tau_n \in [\tau_{\mins}, \tau_{\maxs}]$ 
\begin{align*}
E(u^{n+1}) \le E(u^n) \quad \mbox{for all } n\ge 0
\qquad
\mbox{and}
\qquad
\lim_{n\rightarrow \infty} \| u^n - \uast \|_{H^1(\D)}  \,\, =\,\, 0,
\end{align*}
where $\uast \in \mathbb{S}$ is the unique positive minimizer of problem \eqref{energy-minimization-problem-linear-setting}. The minimum bound $\tau_{\mins}$ for the step size can be arbitrarily close to zero, it is just important that the step sizes do not degenerate.
\end{theorem}
The theorem fully justifies the applicability of the metric-driven scheme to the linear model problem \eqref{RSG-linear}. However, there is another interesting aspect, which relates the scheme to a classical solution algorithm for linear eigenvalue problems: the inverse power iteration. In fact, selecting uniformly $\tau_n=1$ in \eqref{RSG-linear}, we recover the scheme
\begin{eqnarray}
\label{inverse-power-iteration}
u^{n+1} &=&  \frac{ \hspace{-22pt}\mathcal{L}^{-1} u^n }{\|  \mathcal{L}^{-1} u^n \|_{L^2(\D)}}.
\end{eqnarray}
The above inverse iteration converges, for any nonnegative starting value, globally to an $L^2$-normalized eigenfunction $\uast$ to the smallest eigenvalue $\lambda$ of $\mathcal{L} \uast = \lambda \uast$, i.e., eigenvalue problem \eqref{eigenvalue-problem-linear-setting}. As mentioned before, this eigenfunction coincides with the unique positive minimizer of the energy $E$ on $\mathbb{S}$ and we recover the global convergence predicted in Theorem \ref{theorem:global-convergence-linear-setting}. This shows that the gradient flow perspective and the eigenvalue problem perspective are closely related, also in terms of numerical methods. In fact, in the light of the above discussion, the gradient method \eqref{RSG-linear} can be viewed as an inverse power iteration with damping, where $\tau_n$ takes the role of a damping parameter. 

Furthermore, we can recall a classical result for the inverse power iteration, which states that the convergence rate is linear and depends on the first spectral gap of the operator $\mathcal{L}$. Applied to our setting, we obtain that the iterates \eqref{inverse-power-iteration} fulfill 
\begin{align}
\label{estimate-inverse-iteration}
\| \uast - u^n \|_{H^1(\D)} \,\, \lesssim \,\, |\tfrac{\lambda_1}{\lambda_2}|^n \,\, \| \uast - u^0 \|_{H^1(\D)},
\end{align}
where $u^0 \in \mathbb{S}$ is a nonnegative starting value and $\uast$ the positive $L^2$-normalized eigenfunction to the smallest eigenvalue $\lambda_1=\lambda$ of $\mathcal{L}$. As before, $\lambda_2>\lambda_1$ denotes the second smallest eigenvalue of  $\mathcal{L}$. Hence, the larger the spectral gap between $\lambda_1$ and $\lambda_2$, the faster the convergence.

The estimate \eqref{estimate-inverse-iteration} does not only predict that the convergence of the $X$-Sobolev gradient method is heavily influenced by the size of the spectral gap $\tfrac{\lambda_1}{\lambda_2}$, but we also recover the same rate for the continuous gradient flow $u(t)$ given by \eqref{label-gradient-flow-linear-problem} which approaches $\uast$ asymptotically with the rate $ \exp( - (1 - \tfrac{\lambda_1}{\lambda_2}) t )$ (cf. \eqref{continuous-grad-flow-linear-problem-rate}).

The real potential of metric-driven steepest descents unfolds for more general energy minimization problems for which the Euler--Lagrange equations become nonlinear and standard approaches are no longer applicable. Before turning towards such more complicated problems we stay in the linear setting and consider a second important aspect of a practical numerical method, that is, the spatial discretization of the ideal iterations \eqref{RSG-linear}. This is addressed in the next subsection. 

\subsection{Minimization in metric-driven approximation spaces}
\label{subsection-LOD-linear}
In the next step, we want to discuss a spatial discretization of the gradient method \eqref{RSG-linear} in such a way that it respects the particular metric-dependent structure. 

In the following, we let $\mathcal{T}_H$ denote a quasi-uniform and shape-regular triangulation of the polygonal computational domain $\D \subset \R^d$. On $\mathcal{T}_H$, we introduce the classical $\mathbb{P}^1$-Lagrange finite element space of $H^1$-conforming and $\mathcal{T}_H$-piecewise linear functions by
\begin{align}
\label{P1-FEM-space}
V_H \,\,:=\,\, \{ v_H \in C^0(\overline{\D}) \,\, | \,\, v_H \vert_T \in \mathbb{P}^1(T) \,\,\mbox{for all } T \in \mathcal{T}_H \}. 
\end{align}
\subsubsection{Derivation of a metric-driven approximation space}
\label{subsec:derivation-detric-driven-space}

Rather than using a standard finite element space, such as $V_H$ given in \eqref{P1-FEM-space}, 
whose approximation properties typically rely on additional smoothness of
the minimizer $u^\ast$, we aim to construct an approximation space that is
adapted to the problem through the metric 
\begin{align*}
(u,v)_X \,\, := \,\, ( A \nabla u , \nabla v )_{L^2(\D)} + ( V\, u , v )_{L^2(\D)}.
\end{align*}
Recall that the corresponding metric-driven gradient descent \eqref{RSG-linear} computes, as a first step,
\begin{align*}
u^1_{\pr} \,=\, (1-\tau) u^0 
\,+ \,\tau \, ( \mathcal{L}^{-1}u^0,u^0)_{L^2(\mathcal{D})}^{-1}\, \mathcal{L}^{-1}u^0,
\end{align*}
followed by normalization $u^1:=  \frac{ u^1_{\pr} }{\| u^1_{\pr} \|_{L^2(\mathcal{D})}}$. Choosing $\tau=1$ and taking $u^0 \in V_H$ as a classical finite element yields 
\begin{align}\label{updateLODspace}
u^1 = \frac{\mathcal{L}^{-1}u^0}{\|\mathcal{L}^{-1}u^0\|_{L^2(\mathcal{D})}}.
\end{align}
Thus, the iterate  $u^1$ incorporates problem information through the metric $(\cdot,\cdot)_X$, since the update involves the induced solution operator $\mathcal{L}^{-1}$. 
It also coincides with the result of the first step of the inverse iteration \eqref{inverse-power-iteration}. According to the error estimate \eqref{estimate-inverse-iteration}, the approximation of the minimizer by $u^1$ therefore improves compared to the approximation by $u^0$ by a factor $\lambda_1/\lambda_2$. 

An even better estimate can be achieved for a suitable choice of the initial function $u^0$, as can be seen from a block inverse iteration argument. Let $\{u_j\}_{j\ge 1}$ denote an $L^2(\mathcal D)$-orthonormal basis
of eigenfunctions of $\mathcal{L}$ with the corresponding eigenvalues
$0<\lambda_1<\lambda_2\le \lambda_3 \le \dots\,$.
For any $u^0 \in V_H$ with spectral expansion
\begin{align*}
u^0 = \sum_{j=1}^\infty c_j u_j,
\end{align*}
the update \eqref{updateLODspace} yields, up to normalization, 
\begin{align*}
u^1 = \mathcal{L}^{-1}u^0
= \sum_{j=1}^\infty \tfrac{c_j}{\lambda_j} u_j.
\end{align*}
For $c_1\not=0$, which is for example fulfilled for any non-negative $u^0 \not\equiv 0$, 
we can factor out the leading coefficient to obtain
\begin{align*}
u^1
=
\frac{c_1}{\lambda_1}
\left(
u_1
+
\sum_{j=2}^\infty
\frac{c_j}{c_1}\,\frac{\lambda_1}{\lambda_j}\,u_j
\right)
=
\frac{c_1}{\lambda_1}
\left(
u^{\ast}
+
\sum_{j=2}^\infty
\frac{c_j}{c_1}\,\frac{\lambda_1}{\lambda_j}\,u_j
\right).
\end{align*}
After appropriate scaling, the common prefactor $c_1/\lambda_1$ cancels and we obtain
\begin{align}
\label{error-u1-scaled-uast}
 u^1_{\mathrm{scaled}} - u^\ast
=
\sum_{j=2}^\infty
\frac{c_j}{c_1}\,\frac{\lambda_1}{\lambda_j}\,u_j,
\end{align}
where $u^1_{\mathrm{scaled}} := \frac{\lambda_1}{c_1} u^1$ differs from $u^1$ only by a scalar factor. 
In particular, both functions yield the same $L^2$-normalized approximation, i.e.
\begin{align*}
\frac{u^1_{\mathrm{scaled}}}{\| u^1_{\mathrm{scaled}} \|_{L^2(\D)}} =
\frac{u^1}{\| u^1 \|_{L^2(\D)}}.
\end{align*}
Hence, in the light of \eqref{error-u1-scaled-uast}, all higher modes $j\ge 2$ represent the error with respect to $u^\ast$, and their contribution is damped relative to the leading mode by the factor $\lambda_1/\lambda_j$. 
In particular, the smallest index $j\ge 2$ with $c_j\neq 0$ determines the dominant error contribution.

In general, we can choose $u^0 \in V_H$ such that its spectral
expansion satisfies $c_j = 0$ for $2 \le j \le N_{\ast}-1$, where 
$N_{\ast} \simeq N:=\dim V_H$ is on the order of the dimension of the finite element space. 
For such a choice, the dominant error contribution after one inverse
iteration stems from the $N_{\ast}$-th eigenfunction and, therefore,
\begin{align*}
\|u^\ast - u^1\|_{H^1(\mathcal D)}
\;\lesssim\;
\frac{\lambda_1}{\lambda_{N_{\ast}}}.
\end{align*}
In summary, among all possible
starting values $u^0 \in V_H$ there exists one such that $u^0 \in \mathbb{S}$ and $(u^0,u^{\ast})_{L^2(\D)} > 0$ and
\begin{align}
\|u^\ast - u^1\|_{H^1(\mathcal{D})} 
\lesssim \frac{\lambda_1}{\lambda_{N_{\ast}}},
\end{align}
where the hidden constant only depends on the angle $(u^0,u^{\ast})_{L^2(\D)}$. Weyl asymptotics for uniformly elliptic second-order operators, 
\begin{align}
\lambda_{N_{\ast}} \simeq \lambda_{N} \simeq N^{2/d}
\quad \text{for large } N,
\end{align}
suggests the scaling 
\begin{align}
\|u^\ast - u^1\|_{H^1(\mathcal{D})} \lesssim N^{-2/d}.
\end{align}
Since $N \simeq H^{-d}$ for quasi-uniform meshes, this formally yields
\begin{align}
\|u^\ast - u^1\|_{H^1(\mathcal{D})} \lesssim H^2.
\end{align}
We refer to Theorem~\ref{theorem:LOD-estimates-linear-setting} below for a rigorous non-asymptotic error bound.
In other words, the metric-driven approximation space
\begin{align*}
V_H^{\mathcal{L}} \,\,:=\,\, \mathcal{L}^{-1}V_H
\end{align*}
approximates $u^\ast$ uniformly well, independently of additional smoothness
assumptions in $u^\ast$ (whereas extra smoothness $u^\ast\in H^{1+s}(\mathcal{D})$, $0<s\leq 1$, would improve the rate to
$H^{2+s}$). Thus, $V_H^{\mathcal{L}}= \mathcal{L}^{-1}V_H$ can be interpreted as an enrichment of the standard space $V_H$ by problem-specific information incorporated in $\mathcal{L}$. Indeed, for any $v_H \in V_H$, the function $\mathcal{L}^{-1}v_H$ is obtained as the solution of an elliptic problem with $v_H$ as a right-hand side and therefore reflects the fine-scale structure of the operator. This explains why the space can capture multiscale features even though its dimension coincides with that of $V_H$.

In conclusion, the above considerations suggest a spatial discretization in which we do not minimize the energy over the standard FE space $V_H$, but rather the metric-driven space $V_H^{\mathcal{L}}$ to obtain
\begin{align}
\label{LOD-minimizer-linear-setting}
u_H^{\mathcal{L}} \,\,=\,\, \underset{v \in V_H^{\mathcal{L}} \cap\mathbb{S}}{\mbox{arg\hspace{1pt}min}} \,\, E(v)
\end{align}
as an approximation to an exact minimizer $\uast \in \mathbb{S}$. 

\subsubsection{Connections to LOD and approximation properties}
The construction above shows that the application of the metric-induced solution operator naturally transforms a standard finite element space into an enriched space that incorporates problem-specific information of the operator $\mathcal{L}$. This observation provides a conceptual link between metric-driven gradient methods and multiscale approximation spaces.

In fact, the space
\[
V_H^{\mathcal{L}} = \mathcal{L}^{-1}V_H
\]
coincides with a well-established construction in numerical multiscale analysis, namely the Localized Orthogonal Decomposition (LOD) space \cite{MaP14,ActaLOD21,LODbook21}. While LOD spaces are typically derived via orthogonal decompositions and corrector problems, it was observed in \cite{HauckPeterseim23} that they admit the equivalent representation $\mathcal{L}^{-1}V_H$, which is exploited there for the design of suitable  strategies for computing localized basis functions of the LOD space. From the present perspective, this representation reveals that LOD spaces can be interpreted as metric-driven approximation spaces.

For completeness, we briefly sketch the classical LOD representation of $V_H^{\mathcal{L}}$ in the following lemma and afterwards discuss the practical implications.
\begin{lemma}
Let $P_H : H^1_0(\D) \rightarrow V_H$ denote the $L^2$-projection given by
$$
(P_H(u) , v_H )_{L^2(\D)} = ( u , v_H )_{L^2(\D)} \qquad \mbox{for all } v_H \in V_H
$$
and define the kernel of $P_H$ as the detail space
\begin{align*}
W \,\, := \,\, \{ \, v \in H^1_0(\D) \,\, | \,\, P_H(v) = 0 \, \}.  
\end{align*}
With this, the corrector operator $\mathcal{C} : V_H \rightarrow W$ is given, for $v_H \in V_H$, by $\mathcal{C}v_H \in W$ with
\begin{eqnarray*}
(\mathcal{C}v_H , w )_X &=& (v_H,w)_X \qquad \mbox{for all } w\in W,
\end{eqnarray*}
where $(v, w)_X = ( A \nabla v , \nabla w )_{L^2(\D)} + ( V\, v , w )_{L^2(\D)}$ (cf. \eqref{optimal-X-metric-linear}). It holds
\begin{eqnarray*}
V_H^{\mathcal{L}} &=& (I-\mathcal{C})V_H \,\,\, = \,\,\, \{ \, v_H -\mathcal{C} v_H \,\, |\,\, v_H \in V_H \,\}.
\end{eqnarray*}
\end{lemma}
\begin{proof}
The  proof adopts the arguments from \cite{HauckPeterseim23}. First of all, we note that $W$ is a closed subspace of $H^1_0(\D)$ because the kernel of the $L^2$-projection can be equivalently expressed as the kernel of an $H^1$-stable quasi-interpolation operator \cite{Car99,MaP14}. Since $W$ is closed, Lax-Milgram ensures existence and uniqueness of $\mathcal{C}v_H$ for each $v_H$.  Since $W \cap V_H = \{ 0 \}$ we conclude that $\mbox{dim}\, (I-\mathcal{C})V_H = \mbox{dim}\,\mathcal{L}^{-1}V_H$. Hence, it only remains to show the inclusion $\mathcal{L}^{-1}V_H \subset (I-\mathcal{C})V_H$ to conclude that the spaces are identical. For this, consider $\mathcal{L}^{-1}v_H$ for arbitrary $v_H\in V_H$. By the $L^2$-orthogonality of $V_H$ and $W$ and the definition of $\mathcal{L}^{-1}$ it holds
\begin{align*}
( \mathcal{L}^{-1}v_H , w )_X \,\,\,= \,\,\, (v_H , w )_{L^2(\D)} \,\,\,= \,\,\, 0
\qquad \mbox{for all } w \in W.
\end{align*}
Now, define $\tilde{v}_H := P_H(\mathcal{L}^{-1}v_H)$, then $P_H(\mathcal{L}^{-1}v_H-(I-\mathcal{C})\tilde{v}_H)=0$ and hence we have that $\mathcal{L}^{-1}v_H-(I-\mathcal{C})\tilde{v}_H \in W$. Since both $( \mathcal{L}^{-1}v_H , w )_X=0$ and $( (I-\mathcal{C})\tilde{v}_H , w )_X=0$ for any $w\in W$ we conclude that $\mathcal{L}^{-1}v_H=(I-\mathcal{C})\tilde{v}_H$ and therefore $\mathcal{L}^{-1}V_H \subset (I-\mathcal{C})V_H$.
\end{proof}
The classical LOD characterization of $V_H^{\mathcal{L}}$ as $(I-\mathcal{C})V_H$ is important to understand how to obtain a practical basis for the space. In fact, it can be proved that the Green's function associated with the corrector operator $\mathcal{C}$ shows an exponential decay in units of the mesh size $H$. This implies the existence of quasi-local basis functions that can be cheaply computed and used in practical computations. For example, consider a mesh vortex $z$ and a standard nodal hat function $\phi_z\in V_H$, then a corresponding shape function in $V_H^{\mathcal{L}}$ is given by $(I-\mathcal{C})\phi_z$. This shape function shows an exponential decay outside of the nodal patch $\,\omega_z := \mbox{supp}\,\phi_z$, which justifies to truncate it on a small neighborhood of $\omega_z$. The set of all nodal shape functions $(I-\mathcal{C})\phi_z$ forms an ideal (global) basis of $V_H^{\mathcal{L}}$ and the set of suitably truncated shape functions forms an approximate basis that is localized. We call such a basis an LOD basis. For details on the practical realization of the truncation and the efficient algorithmic computation of LOD basis functions using correctors, we refer to \cite{EHMP19}. Alternative superlocalization strategies which do not require correctors but instead exploit optimization techniques are presented in \cite{HauckPeterseim23}.

With the insight that $V_H^{\mathcal{L}}$ can be used in practice by suitable localization strategies, we now turn to its approximation properties which were extensively studied in the context of multiscale problems. The following approximation result can be found for linear eigenvalue problems in \cite{MaP15} and in generalized form in \cite{HMP14b,HePer23}. It is a rigorous extension of the informal derivation given in Section \ref{subsec:derivation-detric-driven-space}.

\begin{theorem}
\label{theorem:LOD-estimates-linear-setting}
Let $\uast\in \mathbb{S}$ denote the exact minimizer of \eqref{energy-minimization-problem-linear-setting} and $u_H^{\mathcal{L}} \in V_H^{\mathcal{L}} \cap\mathbb{S}$ the corresponding minimizer in the metric driven space $V_H^{\mathcal{L}}$ according to \eqref{LOD-minimizer-linear-setting}. Further assume that the signs of the two minimizers are consistent, i.e., $(\uast,u_H^{\mathcal{L}})_{L^2(\D)} \ge 0$. Then, under the minimal regularity assumptions $A \in L^{\infty}(\D, \mathbb{R}^{d \times d})$ and $V\in L^{\infty}(\D,\mathbb{R}_{\ge0})$ it holds
\begin{align}
\label{LOD-est-1-linear-setting}
\| \uast - u_H^{\mathcal{L}} \|_{L^2(\D)} \,+\, H \,\| \uast - u_H^{\mathcal{L}} \|_{H^1(\D)}
\,\,\lesssim\,\, H^3
\qquad\mbox{and}
\qquad 
|E(\uast) - E(u_H^{\mathcal{L}})| \,\,\lesssim\,\, H^4,
\end{align}
where the hidden constants only depend on the maximum values of $A$ and $V$. If $A$ is additionally Lipschitz-continuous, that is $A\in W^{1,\infty}(\D,\mathbb{R}^{d \times d})$, then we have $\uast \in H^2(\D)$ and the estimate improves to
\begin{align}
\label{LOD-est-2-linear-setting}
\| \uast - u_H^{\mathcal{L}} \|_{L^2(\D)} \,+\, H \,\| \uast - u_H^{\mathcal{L}} \|_{H^1(\D)}
\,\,\lesssim\,\, H^4
\qquad\mbox{and}
\qquad 
|E(\uast) - E(u_H^{\mathcal{L}})| \,\,\lesssim\,\, H^6,
\end{align}
where the hidden constant now also depends on $\| A \|_{W^{1,\infty}(\D)}$.\\[0.3em]
Note that the eigenvalue $\lambda$ to the eigenfunction $\uast$ is given by $\lambda = 2\, E(\uast)$, such that the energy estimates translate into equivalent estimates for the eigenvalue error.
\end{theorem}
Recalling that the dimension of the LOD space $V_H^{\mathcal{L}}$ and the standard $\mathbb{P}^1$-finite element space $V_H$ are equal, we observe that the metric-driven choice increases the approximation properties tremendously. Let us discuss the estimates \eqref{LOD-est-1-linear-setting} and \eqref{LOD-est-2-linear-setting} individually, starting with the latter one. If $A\in W^{1,\infty}(\D,\mathbb{R}^{d \times d})$ such that the exact solution $\uast$ admits full $H^2$-regularity, then \eqref{LOD-est-2-linear-setting} ensures that the metric-driven approximation is superconvergent with order $O(H^3)$ in the $H^1$-norm and even with order $O(H^4)$ in the $L^2$-norm. The energy/eigenvalue is even approached with the rate $O(H^6)$. On the contrary, a classical $\mathbb{P}^1$ approximation $u_H\in V_H \cap \mathbb{S}$ will only achieve the standard convergence rates with
\begin{eqnarray*}
\| \uast - u_H \|_{L^2(\D)} \,+\, H \,\| \uast -u_H \|_{H^1(\D)}
&\lesssim& H^2
\qquad\mbox{and}
\qquad 
|E(\uast) - E(u_H)| \,\,\lesssim\,\, H^2,
\end{eqnarray*}
which are significantly worse.

The improvements become even more pronounced when considering estimate \eqref{LOD-est-1-linear-setting}. If the coefficients, especially $A$, are rough and rapidly varying (i.e. multiscale coefficients), then $H^2$-regularity for $\uast$ is not available and rates for the classical $\mathbb{P}^1$-finite element method can degenerate to an arbitrarily slow convergence. Additionally, meaningful approximations in $V_H$ are only obtained if the mesh size resolves all the fine-scale variations of $A$, which can impose a very strong restriction and requires very fine computational meshes. For that reason, solving the problem with a standard approach can become prohibitively expensive. On the other hand, the rates \eqref{LOD-est-1-linear-setting} obtained in the space $V_H^{\mathcal{L}}$ are not only still quadratic in the $H^1$-error (and cubic in the $L^2$-error), but they also do not require that the mesh size resolves the variations of the multiscale coefficients $A$ and $V$. In fact, information about these variations is naturally built into the space $V_H^{\mathcal{L}}$ and coarse mesh sizes $H$ can be used to get meaningful approximations. 

We can conclude that the metric-driven space $V_H^{\mathcal{L}}$ overcomes the issues of classical finite element approximation spaces, especially for problems with multiscale coefficients and low regularity.

\subsubsection{Combined metric-driven gradient-descent}
In order to formulate a fully-discrete method that combines the metric-driven gradient descent with the corresponding metric-driven approximation space $V_H^{\mathcal{L}}$ we introduce a discrete approximation of $\mathcal{L}^{-1}$ from \eqref{def-Linv} by the operator
\begin{align*}
\mathcal{L}_H^{-1} :V_H^{\mathcal{L}}  \to V_H^{\mathcal{L}} 
\end{align*}
with 
\begin{align}
\label{def-Linv-H}
( \mathcal{L}_H^{-1} u_H^{\mathcal{L}} , v_H^{\mathcal{L}} )_X
= ( u_H^{\mathcal{L}} , v_H^{\mathcal{L}} )_{L^2(\mathcal{D})}
\qquad \text{for all } v_H^{\mathcal{L}} \in V_H^{\mathcal{L}}.
\end{align}
With this, we can formulate the metric-driven descent \eqref{RSG-linear} directly in $V_H^{\mathcal{L}}$. 
\begin{definition}[Metric-driven discretization]\quad\\
\label{definition-metric-adaptive-disc-1-new}
Given $u_H^{\mathcal{L},n} \in V_H^{\mathcal{L}} \cap \mathbb{S}$ the next iterate $u_H^{\mathcal{L},n+1} \in V_H^{\mathcal{L}} \cap \mathbb{S}$ is given by
$$
u_{H}^{\mathcal{L},n+1} := \frac{\hspace{-22pt}u_{H,\pr}^{\mathcal{L},n+1}}{\| u_{H,\pr}^{\mathcal{L},n+1} \|_{L^2(\D)}}, 
$$
where $u_{H,\pr}^{\mathcal{L},n+1}\in V_H^{\mathcal{L}}$ is defined by
\begin{eqnarray}
\label{metric-adaptive-method-1}
u_{H,\pr}^{\mathcal{L},n+1} 
\,&=&\,  (1 - \tau_n)\, u_H^{\mathcal{L},n}  \,+\, \tau_n \, ( \mathcal{L}^{-1}_H u_H^{\mathcal{L},n} ,u_H^{\mathcal{L},n} )_{L^2(\D)}^{-1} \, \mathcal{L}^{-1}_H u_H^{\mathcal{L},n},
\end{eqnarray}
where $\mathcal{L}^{-1}_H u_H^{\mathcal{L},n}$ is the approximate Riesz representation of $u_H^{\mathcal{L},n}$ in the $X$-metric according to \eqref{def-Linv-H}.
\end{definition}
The above construction combines the two key ingredients of the metric-driven approach: on the one hand, the choice of the metric induces a problem-adapted approximation space $V_H^{\mathcal L}$, and on the other hand, it determines the form of the iterative scheme through the corresponding Sobolev gradients. In this sense, the discretization is fully consistent with the underlying metric structure, as both the search space and the descent direction are derived from the same operator-induced geometry.

From a practical point of view, this leads to an iterative scheme that can be interpreted as a projected inverse iteration within the multiscale space $V_H^{\mathcal L}$. Compared to classical approaches, where discretization and iteration are designed independently, the metric-driven formulation provides a unified framework in which both aspects are intrinsically coupled.

Each iteration of the scheme \eqref{metric-adaptive-method-1} requires the computation of a Riesz representative with respect to the $X$-metric restricted to $V_H^{\mathcal L}$, which corresponds to solving a linear problem in the multiscale space. Due to the LOD structure of $V_H^{\mathcal{L}}$, this step can be realized efficiently using localized basis functions, making the overall scheme computationally feasible.

\section{Metric-driven discretizations for the Gross--Pitaevskii equation}
\label{section:metric-adaptive-GPE}

In the previous section, a quadratic minimization problem was considered, resulting in the approximation of an eigenfunction to the smallest eigenvalue of a linear eigenvalue problem. In this section, we want to consider a more complicated setup which involves a nonlinear eigenvalue problem and we want to sketch how our previous considerations generalize to such a setting and what we can expect in terms of approximation properties.

In the following we consider the so-called Gross--Pitaevskii functional which takes an important role in many physical applications, especially when describing the ground states (i.e. lowest energy states) of quantum systems \cite{PiS03}. To introduce the functional we make the same assumptions as in the previous section, i.e., $\D\subset \R^d$ (for $d\in \{1,2,3\}$) is a polygonal Lipschitz-domain, $A \in L^{\infty}(\D, \mathbb{R}^{d \times d})$ is an almost everywhere symmetric and uniformly elliptic coefficient, $V\in L^{\infty}(\D,\mathbb{R}_{\ge0})$ is a nonnegative potential term and, additionally, $\beta\ge0$ is a constant parameter (which usually models the strength of repulsive particle interactions).

With this, the Gross--Pitaevskii functional $E:H^1_0(\D) \rightarrow \R_{\ge0}$ is given by
\begin{align*}
E(u ) := \frac{1}{2} \int_{\D} |A^{1/2} \nabla u|^2 + V |u|^2 + \frac{\beta}{2} |u|^4 \,\,\mbox{d}x
\end{align*}
and we are again concerned with finding an $L^2$-normalized minimizer 
\begin{align}
\label{energy-minimization-problem-GPE-setting}
\uast \,=\, \underset{ v \in \mathbb{S}}{\mbox{arg\hspace{2pt}min}} \, E(v).
\end{align}
The analytical properties of the problem do not change much compared to the linear setting in Section \ref{section:2}. In particular, the minimizer $\uast$ in \eqref{energy-minimization-problem-GPE-setting} is still unique up to sign and either strictly positive or strictly negative in the interior $\D$, cf.~\cite{CCM10}.
From the Euler--Lagrange equations $E^{\prime}(\uast) = \lambda \,( \uast ,\cdot )_{L^2(\D)}$ and the formula
\begin{align}
\label{formula-deriv-energy-GPE}
\langle E^{\prime}(u) , v \rangle \,\,\, = \,\,\, ( A \nabla u , \nabla v )_{L^2(\D)} + ( V u , v)_{L^2(\D)} + ( \beta |u|^2 u , v)_{L^2(\D)},
\end{align}
we obtain the nonlinear eigenvalue problem known as the {\it Gross--Pitaevskii equation} (GPE). The GPE seeks the smallest eigenvalue $\lambda \in \mathbb{R}_{>0}$ and corresponding eigenfunction $\uast \in \mathbb{S}$ such that 
\begin{align}
\label{eigenvalue-problem-GPE-setting}
( A \nabla \uast , \nabla v )_{L^2(\D)} + ( V \uast , v)_{L^2(\D)} + ( \beta |\uast|^2 \uast , v)_{L^2(\D)}  \, =\, \lambda \, (\uast, v)_{L^2(\D)}
\end{align}
for all $v\in H^1_0(\D)$. The fact that the smallest eigenvalue $\lambda$ of \eqref{eigenvalue-problem-GPE-setting} belongs to a minimizer of \eqref{energy-minimization-problem-GPE-setting} is not trivial. A proof is given in \cite{CCM10}.
Also note that the choice of the normalization (i.e. $\uast \in \mathbb{S}$ or equivalently $\| \uast \|_{L^2(\D)}=1$) is important in the nonlinear setting and that different normalizations would change the eigenvalues of \eqref{eigenvalue-problem-GPE-setting} as well as the structure of the eigenfunctions. 

Next, let us again discuss how metric-driven schemes can be applied to solve the problem. As in the linear setting, we discuss the construction of iterative solvers and the choice of the spatial discretization individually. 

\subsection{Metric-driven steepest descents for the GPE}
\label{subsection:metric-steepest-descent-GPE}
Since the eigenvalue problem \eqref{eigenvalue-problem-GPE-setting} is nonlinear, a direct application of linear eigenvalue solvers (such as the inverse iteration) is not possible anymore.  However, we can still apply the abstract framework from Section \ref{subsection-metric-driven-grad-flow-linear} by exploiting the concept of Riemannian gradient methods in combination with metric-driven Sobolev gradients.

To avoid unnecessary repetitions, we skip the introduction of a preliminary (metric-driven) gradient flow and directly turn to the formulation of a gradient method. Following the general descriptions from the linear setting and in particular \eqref{Rie-grad-method-linear}, a Riemannian gradient method in a given $X$-metric (induced by an inner product $(\cdot,\cdot)_X$) is of the form 
\begin{eqnarray}
\label{Rie-grad-method-GPE}
u^{n+1} &=& \frac{ \hspace{-23pt}u^n \, - \, \tau_n ( P_{u^n,X} \circ \nabla_X E)u^n }{\| u^n \, - \, \tau_n ( P_{u^n,X} \circ \nabla_X E)u^n  \|_{L^2(\D)}},
\end{eqnarray}
where we recall $(P_{u^n,X} \circ \nabla_X E)u^n$ as the Riemannian Sobolev gradient composed from the Sobolev gradient $\nabla_X E(u) \in X$ and its $X$-orthogonal projection $P_{u^n,X}$ into the tangent space $T_u \mathbb{S}$ given by \eqref{orth-proj-tangent-space}. As in the linear setting, we can ask if we can select the metric $(\cdot,\cdot)_X$ in such a way that $\nabla_X E(u)$ becomes the identity. In the linear setting, we could find a universal inner product that yields this property, but in the nonlinear setting of the GPE, the inner product needs to depend on the location $u \in \mathbb{S}$ where we evaluate the regular gradient $E^{\prime}(u)$ given by \eqref{formula-deriv-energy-GPE}. For a linearization point $u \in \mathbb{S}$, we therefore define the linearized gradient $\mathcal{L}_u$ such that
\begin{align}
\label{def-Lu-GPE}
\langle \mathcal{L}_u v , w \rangle \,\,\, :=  \,\,\, 
( A \nabla v , \nabla w )_{L^2(\D)} + ( V v , w)_{L^2(\D)} + ( \beta |u|^2 v , w )_{L^2(\D)}.
\end{align}
Obviously, in the light of \eqref{formula-deriv-energy-GPE}, $E^{\prime}(u)=\mathcal{L}_u u$ holds as an identity in $H^{-1}(\D)$. We can therefore use the inner product  $(\cdot,\cdot)_{X_u}:=\langle \mathcal{L}_u \cdot , \cdot \rangle$ in $H^1_0(\D)$ to obtain a Sobolev gradient with the desired feature $\nabla_{X_u} E(u)=u$. This identity can be verified by exploiting the definition of $\nabla_{X_u} E(u)$ in \eqref{def-Sobolev-gradient} which yields $\nabla_{X_u} E(u)\in H^1_0(\D)$ as the solution to
\begin{align*}
(\nabla_{X_u} E(u) , v )_{X_u} \, = \,
\langle E^{\prime}(u) , v \rangle  \, = \, \langle \mathcal{L}_u u , v \rangle\, = \, 
( u , v )_{X_u}
\end{align*}
for all $v\in H^1_0(\D)$ and hence $\nabla_{X_u} E(u)=u$. Again, as in the linear setting, the corresponding projection $P_{u,X_u}$ into the tangent space can be computed explicitly as 
 \begin{align*}
 P_{u,X_u}(v) = v - \frac{(u,v)_{L^2(\D)}  }{ ( \mathcal{L}_{u}^{-1} u, u )_{L^2(\D)} }   \mathcal{L}_{u}^{-1} u,
 \end{align*}
 where $\mathcal{L}_{u}^{-1} : H^1_0(\D) \rightarrow H^1_0(\D)$ is the $u$-linearized solution operator such that for any $z\in H^1_0(\D)$ the image $\mathcal{L}_{u}^{-1}z \in H^1_0(\D)$ fulfills
 \begin{align}
 \label{Lu-inverse-GPE}
\langle \mathcal{L}_u \, (\mathcal{L}_{u}^{-1}z) , v \rangle \,\, = \,\, (z,v)_{L^2(\D)}
\qquad \mbox{for all } v\in H^1_0(\D).
 \end{align}
Together with $\nabla_{X_u} E(u)=u$ we obtain the following realization of the metric-driven Riemannian gradient method \eqref{Rie-grad-method-GPE}:
\begin{definition}[Metric-driven Riemannian gradient method for the GPE]
\label{definition-metric-driven-grad-method-GPE}
For a starting value $u_0 \in \mathbb{S}$ with $u_0 \ge 0$ and a sequence of (pseudo) time step sizes $\tau_n>0$ and for $n\ge 0$, the iterations of the metric-driven Riemannian gradient method are given by
\begin{eqnarray}
\label{RSG-GPE}
u^{n+1}(\tau) &:=& \frac{ \hspace{-23pt}(1 - \tau)\,u^n  \,+\, \tau \, ( \mathcal{L}_{u^n}^{-1} u^n , u^n )_{L^2(\D)}^{-1} \,\mathcal{L}_{u^n}^{-1} u^n }{\| (1 - \tau)\,u^n  \,+\, \tau \, ( \mathcal{L}_{u^n}^{-1} u^n , u^n )_{L^2(\D)}^{-1} \,\mathcal{L}_{u^n}^{-1} u^n \|_{L^2(\D)} } 
\end{eqnarray}
and we set $u^{n+1}=u^{n+1}(\tau^n)$ for the optimal step size
\begin{eqnarray*}
\tau_n &=&
\underset{0<\tau \le 2}{\mbox{\normalfont arg\hspace{1pt}min}} \,\, E\hspace{-1pt}\left( u^{n+1}(\tau) \right).
\end{eqnarray*} 
\end{definition}
The scheme can be interpreted as a generalized inverse iteration with damping parameter $\tau$. This is seen by noticing that for $\tau=1$, the scheme \eqref{RSG-GPE} reduces to
\begin{eqnarray*}
u^{n+1} &:=& \frac{ \hspace{-23pt}\mathcal{L}_{u^n}^{-1} u^n }{\| \mathcal{L}_{u^n}^{-1} u^n \|_{L^2(\D)} },
\end{eqnarray*}
i.e., a direct generalization of the inverse iteration to nonlinear eigenvalue problems. 

The gradient method in Definition \ref{definition-metric-driven-grad-method-GPE} is well-studied in the literature, cf. \cite{101093imanumdraf046,APS23Newton,PH24,HeP20,PHMY242,Zhang2022} for various related settings, and can be even generalized to conjugate gradients \cite{AHYY24}. The following result can be found in \cite[Theorem 5.1]{HeP20}.
\begin{theorem}[Global convergence]
Consider the metric-driven Riemannian gradient method in \eqref{RSG-GPE} for arbitrary step size $\tau_n$ and a starting value $u_0 \in \mathbb{S}$ which satisfies $u_0\ge0$. There exists a step size interval $[\tau_{\mins}, \tau_{\maxs}] \subset (0,2)$ such that for all $\tau_n \in [\tau_{\mins}, \tau_{\maxs}]$ 
\begin{align*}
E(u^{n+1}) \le E(u^n) \quad \mbox{for all } n\ge 0
\qquad
\mbox{and}
\qquad
\lim_{n\rightarrow \infty} \| u^n - \uast \|_{H^1(\D)}  \,\, =\,\, 0,
\end{align*}
where $\uast \in \mathbb{S}$ is the unique positive minimizer of  \eqref{energy-minimization-problem-GPE-setting}. In particular, the method is globally convergent if the optimal step size $\tau_n$ is selected.
\end{theorem}
The local convergence rates of the Riemannian gradient method depend on the choice of the step size and again on the size of spectral gaps. To present an analogous result to the linear case in \eqref{estimate-inverse-iteration}, let us consider again the choice $\tau=1$, where the method \eqref{RSG-GPE} becomes a direct generalization of the inverse iteration. The following theorem can be found in \cite{PH24}.
\begin{theorem}[Local convergence]
\label{theorem:local-convergence-GPE}
Consider the metric-driven method in \eqref{RSG-GPE} for a uniform step size $\tau_n=1$ and let $\uast \in \mathbb{S}$ be the positive global minimizer to \eqref{energy-minimization-problem-GPE-setting}. Then, for any $\eps>0$ there exists an open neighborhood $U_{\eps} \subset H^1_0(\D)$ of $\uast$, such that for any starting value $u_0 \in  U_{\eps} \cap \mathbb{S}$ it holds
\begin{align}
\label{estimate-inverse-iteration-GPE}
\| \uast - u^n \|_{H^1(\D)} \,\, \lesssim \,\, C_{\eps} \, |\tfrac{\lambda_1}{\lambda_2} + \eps|^n \,\, \| \uast - u^0 \|_{H^1(\D)}.
\end{align}
Here, $C_{\eps}>0$ is a generic, but $\eps$-dependent constant and $\lambda_1=\lambda$ and $\lambda_2>\lambda_1$ are the smallest and the second smallest eigenvalue of the linearized GPE that seeks the eigenpairs $(u_i,\lambda_i) \in \mathbb{S} \times \R_{>0}$ to
\begin{align*}
\langle \mathcal{L}_{\uast} u_i , v \rangle 
\,\,\, = \,\,\, \lambda_i \, (u_i , v)_{L^2(\D)}
\qquad \mbox{for all } v\in H^1_0(\D).
\end{align*}
\end{theorem}
For $\beta=0$, when the GPE collapses into the linear setting, Theorem \ref{theorem:local-convergence-GPE} just recovers the well-known rate \eqref{estimate-inverse-iteration} from the inverse iteration. However, in contrast to the linear setting, the rate in \eqref{estimate-inverse-iteration-GPE} is not sharp for $\beta>0$ and can be refined. A more precise characterization of the asymptotic convergence rates (and for arbitrary $\tau$) is more technical and we refer the interested reader to \cite{PH24,PHMY242,FT26}. 

\begin{remark}[Alternative approaches]
Besides Sobolev gradient-based approaches discussed in this subsection, the Gross--Pitaevskii problem \eqref{eigenvalue-problem-GPE-setting} is commonly approached by a variety of numerical methods, including self-consistent field (SCF) iterations \cite{CaL00,CaL00B,UJR21}, Newton-type methods \cite{APS23Newton,WWB17,XXXY21}, $L^2$-gradient flow inspired schemes \cite{BaD04,FST26,FengTangWangIMA}, as well as nonlinear variants of inverse iteration \cite{JarKM14,AHP21NumMath,PH24}. We refer to \cite{HenJar24} for a recent overview, including a quantitative comparison.

In contrast to these approaches, the metric-driven formulation is not only designed as an alternative standalone solver, but rather as a framework in which both the iterative scheme and the underlying approximation space are derived from a common operator-dependent metric. In particular, the use of Sobolev gradients can be interpreted as an intrinsic preconditioning strategy that is naturally adapted to the problem. This distinguishes the approach from classical methods, where preconditioning and space enrichment are typically introduced as separate design components. Our focus here is on the conceptual role of the metric in the joint design of the iteration and the approximation space, rather than on a detailed algorithmic or performance-oriented comparison of different solution strategies.
\end{remark}

\subsection{Minimization of the GPE functional in metric-driven approximation spaces}
\label{subsection-LOD-space-GPE}
Next, we will turn to the construction of a metric-driven approximation space for the GPE, where we adopt the notation from Section \ref{subsection-LOD-linear}. In particular, $\mathcal{T}_H$ is a quasi-uniform and shape-regular triangulation of $\D \subset \R^d$ and the $\mathbb{P}^1$-Lagrange finite element space on $\mathcal{T}_H$ is denoted by $V_H \subset H^1_0(\D)$. 

For the derivation of the approximation space, we can essentially proceed analogously to the linear setting in Section \ref{subsection-LOD-linear}. However, this would yield a sequence of spaces of the form $\mathcal{L}_{u_H^{\mathcal{L},n}}^{-1} V_H$, where $u_H^{\mathcal{L},n}$ is a previous approximation obtained in an old approximation space $\mathcal{L}_{u_H^{\mathcal{L},n-1}}^{-1} V_H$. The old space $\mathcal{L}_{u_H^{\mathcal{L},n-1}}^{-1} V_H$ itself is based on another previous approximation $u_H^{\mathcal{L},n-1} \in \mathcal{L}_{u_H^{\mathcal{L},n-2}}^{-1} V_H$ and so on. This kind of recursive construction of the approximation space involves the repeated computation and update of basis functions and is therefore impractical and not economical. For that reason, we use the linearization of $\mathcal{L}_u$ around the point $u=0$ to construct the approximation space. It turns out that this is sufficient to obtain the optimal approximation order known from the linear setting. We fix the construction in the following definition.
\begin{definition}[Metric-driven approximation space]
Let $\mathcal{L}_0$ denote the elliptic operator defined in \eqref{def-Lu-GPE} for $u=0$ (or, equivalently, $\beta=0$) and $\mathcal{L}_0^{-1} : H^1_{0}(\D) \rightarrow H^1_0(\D)$ the corresponding solution operator according to \eqref{Lu-inverse-GPE}. We define the metric-driven approximation space by
\begin{align}
\label{LOD-space-GPE}
V_H^{\mathcal{L}_0} \,\, := \,\, \mathcal{L}_0^{-1} V_H,
\end{align}
where $V_H$ is the $\mathbb{P}^1$-FEM space from \eqref{P1-FEM-space}. The metric-driven approximation is now defined as a normalized minimizer of $E$ on $V_H^{\mathcal{L}_0}$, i.e.,  
\begin{align}
\label{LOD-minimizer-GPE}
u_H^{\mathcal{L}_0} \,\,=\,\, \underset{v \,\in\, V_H^{\mathcal{L}_0} \, \cap\,\,\mathbb{S}}{\mbox{\normalfont arg\hspace{1pt}min}} \,\, E(v).
\end{align}
\end{definition}
Corresponding minimizers can be numerically found by performing the iterations \eqref{RSG-GPE} in $V_H^{\mathcal{L}_0}$, analogously as in the linear case.

Again, the metric-driven space in \eqref{LOD-space-GPE} coincides with the LOD space for the GPE and accordingly we have the following approximation results proved in \cite{HMP14b,HePer23}.
\begin{theorem}
\label{theorem:LOD-estimates-GPE}
Let $\uast\in \mathbb{S}$ denote the exact minimizer of \eqref{energy-minimization-problem-GPE-setting} and $u_H^{\mathcal{L}_0} \in V_H^{\mathcal{L}_0} \cap\mathbb{S}$ the corresponding minimizer in \eqref{LOD-minimizer-GPE} such that $(\uast,u_H^{\mathcal{L}_0})_{L^2(\D)} \ge 0$. 

If merely $A \in L^{\infty}(\D, \mathbb{R}^{d \times d})$ and $V\in L^{\infty}(\D,\mathbb{R}_{\ge0})$ it holds
\begin{equation}
\label{LOD-est-1-GPE}
  \begin{aligned}
\| \uast - u_H^{\mathcal{L}_0} \|_{L^2(\D)} \,+\, H \,\| \uast - u_H^{\mathcal{L}_0} \|_{H^1(\D)}
\,\,&\lesssim\,\, H^3 \qquad\mbox{and}\\[0.3em]
|E(\uast) - E(u_H^{\mathcal{L}_0})| \,\,&\lesssim\,\, H^4,
\end{aligned}
\end{equation}
where the hidden constants only depend on $\beta$ and the maximum values of $A$ and $V$.

If $A$ additionally fulfills $A\in W^{1,\infty}(\D,\mathbb{R}^{d \times d})$, then we have $\uast \in H^2(\D)$ and the improved estimates become
\begin{equation}
\label{LOD-est-2-GPE}
  \begin{aligned}
\| \uast - u_H^{\mathcal{L}_0} \|_{L^2(\D)} \,+\, H \,\| \uast - u_H^{\mathcal{L}_0} \|_{H^1(\D)}
 \,\,&\lesssim\,\, H^4
 \qquad\mbox{and}\\[0.3em]
\qquad 
|E(\uast) - E(u_H^{\mathcal{L}_0})|  \,\,&\lesssim\,\, H^6,
\end{aligned}
\end{equation}
for a hidden constant that additionally depends on $\| A \|_{W^{1,\infty}(\D)}$.
\end{theorem}
As we can see, despite using a fixed approximation space $V_H^{\mathcal{L}_0}$ that only uses the linear part of the Gross-Pitaevskii operator $E^{\prime}(u)$, we still obtain the full convergence orders from the linear setting as presented in Theorem \ref{theorem:LOD-estimates-linear-setting}. As before, the results in Theorem \ref{theorem:LOD-estimates-GPE} show an enormous potential in multiscale and low-regularity applications. Similar findings can be made for the time-dependent Gross-Pitaevskii equation \cite{Waernegard1,Waernegard2,Waernegard3} or other types of semi-linear problems \cite{HMP14}. 

A numerical illustration for the approximation properties of $V_H^{\mathcal{L}_0}$ is given in the next section for a more advanced application.

\section{Application to ground states of spin-orbit coupled Bose--Einstein condensates}
\label{section:metric-adaptive-SOBEC}
In this section, we give an advanced example of how a metric-driven approach can be extended to more complicated models. For that, we investigate the numerical computation of ground states of two-component Bose-Einstein condensates (BECs) that are spin-orbit coupled (SO-coupled).\\
A Bose-Einstein condensate is a state of matter that is realized when bosonic gases are cooled to temperatures near absolute zero \cite{Bos24,Ein24,PiS03}. In this state of matter, the atoms are no longer distinguishable from one another and instead form a type of \quotes{superatom}, i.e., all the atoms behave as if they were only a single atom, as they all occupy the same quantum state. We will consider two-component SO-coupled BECs, meaning that by means of so-called Raman lasers, the spin of an atom is tied to its momentum. Our goal is to identify ground states of the SO-coupled BEC. These ground states represent  stable stationary states that minimize the energy of the condensate. The stability of a ground state allows us to examine interesting quantum phenomena. Our numerical experiments will show that, in the considered configuration, the SO-coupled BEC can exhibit a \emph{super solid phase}, which means that it possesses properties that are both inherent to solids (i.e. crystalline structure) and superfluids (i.e. frictionless flow), see \cite{Li17}. 

{\bf Notation.} For a proper introduction of the problem we first need to specify the relevant function spaces. In the following, we denote by $\bfL^2(\D)$ the space 
\begin{eqnarray*}
\bfL^2(\D) &:=&  L^2(\D,\C^2)
\,\,\,:=\,\,\,
\{ \, \bfv = (v_1,v_2) \,\,|\,\, v_i \in L^2(\D,\C), \,\,\, i=1,2 \,\},
\end{eqnarray*}
equipped with the $\R$-inner product 
\begin{eqnarray*}
( \bfv , \bfw )_{\bfL^2(\D)} &:=& \Re \int_{\D} v_1 \overline{w_1} \,\dx \,+\, \Re \int_{\D} v_2 \overline{w_2} \,\dx,
\end{eqnarray*}
where $x = (x_1, x_2)$ for \(\D\subset\mathbb{R}^2\) and \(x=(x_1, x_2, x_3)\) for \(\D\subset\mathbb{R}^3\).
Similarly, we let $\bfH^1_0(\D) := H^1_0(\D,\C^2)$ and define
\begin{eqnarray*}
( \bfv , \bfw )_{\bfH^1(\D)} &:=& \Re \int_{\D} \nabla v_1 \cdot \overline{\nabla w_1} \,\dx \,+\, \Re \int_{\D} \nabla v_2 \cdot \overline{\nabla w_2} \,\dx,
\end{eqnarray*}
which is an inner product by the Poincar\'e inequality.
For readability in the proofs, we sometimes also denote $( \nabla \bfv , \nabla \bfw )_{\bfL^2(\D)}:=( \bfv , \bfw )_{\bfH^1(\D)}$. 
\subsection{Mathematical setting and problem formulation}
We now turn to the problem formulation. As before, we are concerned with the minimization of a functional $E$ on a constraint manifold $\boldS$. In this section, the manifold is used to impose a normalization constraint for the mass of a BEC that consists of two components $u_1$ and $u_2$. Consequently, it is given by
\[
\boldS \,\,:=\,\, \{\, \bfu=(u_1,u_2) \in \bfH_0^1(\D)\,\, | \,\, 
\int_{\D} |u_1|^2 + |u_2|^2 \dx\, = \, 1
\,\}.
\]
The functional $E : \bfH_0^1(\D) \rightarrow \R$ in the case of SO-coupled BECs describes the total energy of a system and is defined as
\begin{align}
 	 E(\bfu)  \,\,&:=\,\, \frac{1}{2}\int_{\D}\sum_{j = 1}^2 \Big(\frac{1}{2} |\nabla u_j|^2 + V_j(x)|u_j|^2\label{eq:SO-energy}
 	\Big) + \frac{\delta}{2}\Big(|u_1|^2 - |u_2|^2\Big) + \Omega \, \Re (u_1\conju{2}) \\
 	\nonumber&\,\,\,+\,\ci k_0\big(\conju{1}\partial_1 u_1 - \conju{2}\partial_1 u_2\big) + \frac{\beta_{11}}{2}|u_1|^4 + \frac{\beta_{22}}{2}|u_2|^4 + \beta_{12}|u_1|^2|u_2|^2\, \dx,
\end{align}
for a tuple \( \bfu = (u_1, u_2) \in \bfH_0^1(\D) \) that describes the quantum states of the two components of the BEC. The functions \(V_1(x)\) and \(V_2(x)\) represent real-valued external trapping potentials that confine the system. The differential operator \(\partial_1\) refers to the partial derivative with respect to the first component, i.e., \(\partial_1 = \partial_{x_1}\). The constants $\beta_{ij}\ge 0$ (for $1\le i,j \le 2$) characterize the interatomic interactions and depend on the type of particles. Furthermore, the constant $\delta$ denotes the detuning parameter associated with the Raman transition, $\Omega$ is the effective Rabi frequency describing the strength of the Raman coupling and \(k_0\) is the wave number of Raman lasers which models the SO-coupling strength (cf.~\cite{BC15}). With this, a ground state (lowest energy state) $\bfu=(u_1,u_2) \in \boldS$ of the SO-coupled BEC is defined as a minimizer of $E$ on $\boldS$, i.e.,
\begin{align}
\label{energy-minimization-problem-SO-coupled-BEC}
\bfu \,=\, \underset{ \bfv \in \boldS}{\mbox{arg\hspace{2pt}min}} \, E(\bfv).
\end{align}
The physical quantity of interest is the (real-valued) density of the components which is given by $|u_1|^2$ and $|u_2|^2$ respectively.  

\subsubsection{Well-posedness}
In order to compute a ground state with a metric-driven approach and to analyze the resulting method, we make a set of assumptions:
\begin{enumerate}[label={(A\arabic*)}]
	\item\label{A1} \(\D\subseteq \mathbb{R}^d\) is a bounded Lipschitz-domain for \(d=2,3\).
	\item\label{A2}  \(V_j\in L^\infty(\D, \mathbb{R}_{\geq 0})\) fulfill  \(V_j(x)-\frac{|\delta| + |\Omega| + 2k_0^2}{2}\geq 0\) for \(j=1,2\).
	\item\label{A3} \(\beta_{11}, \beta_{12}, \beta_{22}\geq 0\) are real-valued and positive (i.e., all interactions are repulsive).
\end{enumerate}
Assumption \ref{A1} is a natural assumption on the computational domain and \ref{A3} is a restriction on the particle types. Assumption \ref{A2} ensures that the trapping potentials are non-negative and sufficiently strong compared to other effects. Ultimately, this guarantees that the energy remains positive. Note that \ref{A2} is in any case uncritical, since we can always shift the energy by an arbitrary constant (such as $\frac{|\delta| + |\Omega| + 2k_0^2}{2}$) without changing the minimizers of $E$ on $\boldS$. This exploits that all minimizers satisfy the $L^2$-normalization constraint.

Mathematically, we require \ref{A2} to select a suitable linearization of \(E^{\prime}(u)\) that induces an inner product/metric in the spirit of Sections \ref{section:2} and \ref{section:metric-adaptive-GPE}. To make the linearization precise, we define the operator \(\mathcal{L}_{\bfu}: \bfH_0^1(\D) \to \bfH_0^1(\D)^{\ast} \) for a linearization point $\bfu\in \bfH_0^1(\D)$  by
\begin{eqnarray}
\label{def-Lu}
\lefteqn{ \langle \mathcal{L}_{\bfu} \bfv, \bfw \rangle \,\,\,=\,\,\, \Re
	\int_{\D} \sum_{j = 1}^2\Big(\frac{1}{2}\nabla v_j \cdot \nabla \conjw{j} + V_jv_j\conjw{j}\Big)
	+ \frac{\delta}{2}\big(v_1\conjw{1} - v_2\conjw{2}\big) + \frac{\Omega}{2}\big(v_1\conjw{2} + v_2\conjw{1}\big)\dx}\\
\nonumber	&\enspace&
\hspace{-15pt} 
+ \,\Re
	\int_{\D} \ci k_0\big(\conjw{1}\partial_1v_1 - \conjw{2}\partial_1 v_2\big) 
   + \beta_{11}|u_1|^2v_1\conjw{1} + \beta_{22}|u_2|^2v_2\conjw{2}
	+ \beta_{12}\big(|u_2|^2v_1\conjw{1} + |u_1|^2v_2\conjw{2}\big)\dx
\end{eqnarray}
for $\bfv,\bfw\in \bfH_0^1(\D)$. 
Computing the Fr\'echet derivative of $E$ confirms that $\mathcal{L}_{\bfu}$ is indeed a linearization of $E^{\prime}(\bfu)$ as
$$
\langle E'(\bfu), \bfv \rangle = \langle \mathcal{L}_{\bfu}\bfu, \bfv \rangle.
$$
If $\beta_{11}=\beta_{12}=\beta_{22}=0$, then $E^{\prime}$ is linear and coincides with the linear operator $\mathcal{L}_{\boldsymbol{0}}$. Hence, the linearization $\mathcal{L}_{\bfu}$ is natural. Also note that due to the presence of the real part in front of the integral in the definition of $ \mathcal{L}_{\bfu}$, the operator is self-adjoint and therefore induces a symmetric bilinear form.

As in the setting of the previous sections, constrained minimizers fulfill the corresponding Euler--Lagrange equations which take the form of a (nonlinear) eigenvalue problem. Exploiting that $E'(\bfu)=\mathcal{L}_{\bfu}\bfu$, a minimizer $\bfu \in \boldS$ to \eqref{energy-minimization-problem-SO-coupled-BEC} is an eigenfunction with eigenvalue $\lambda \in \R_{>0}$ to
\begin{equation}
	\langle \mathcal{L}_{\bfu}\bfu, \bfv \rangle \, = \, \lambda \, ( \bfu , \bfv )_{\bfL^2(\D)} \label{eq:eigenvalue}
\end{equation}
for all $\bfv \in \bfH_0^1(\D)$. Using the definition of $\mathcal{L}_{\bfu}$, the strong form of the eigenvalue problem \eqref{eq:eigenvalue} reads
\begin{equation*}
	\begin{aligned}
		\lambda u_1 & = \Bigg[-\frac{1}{2}\Delta + V_1(x)+ \ci k_0\partial_1 +\frac{\delta}{2}+\big(\beta_{11}|u_1|^2+\beta_{12}|u_2|^2\big)\Bigg]u_1 + \frac{\Omega}{2}u_2\\
		\lambda u_2 & = \Bigg[-\frac{1}{2}\Delta + V_2(x)-\ci k_0\partial_1 -\frac{\delta}{2}+\big(\beta_{12}|u_1|^2+\beta_{22}|u_2|^2\big)\Bigg]u_2 + \frac{\Omega}{2}u_1
	\end{aligned}
\end{equation*}
for $\bfu=(u_1,u_2)$.
Note that the eigenvalue $\lambda$ that belongs to a ground state (i.e.\ a global minimizer fulfilling \eqref{energy-minimization-problem-SO-coupled-BEC}) is not necessarily the smallest eigenvalue of \eqref{eq:eigenvalue}. This is an important difference to the settings from the previous sections.

Another difference compared to the previous settings is that existence and uniqueness of ground states is no longer obvious. The question was investigated in \cite{BC15} and we have the following  well-posedness result (which relies on the non-negativity of the parameters $\beta_{11}, \beta_{12}, \beta_{22}$):
\begin{theorem}[Existence of ground states of SO-coupled BECs]
Assume \ref{A1}-\ref{A3}, then there exists at least one minimizer $\bfu \in \boldS$ to \eqref{energy-minimization-problem-SO-coupled-BEC}. Ground states are at most locally unique up to complex phase shifts $\exp(\ci \omega)$ for any angle $\omega\in [-\pi, \pi)$. This means that if $\bfu \in \boldS$ denotes an arbitrary ground state, then \,\,$\exp(\ci \omega)\bfu \in \boldS$\,\, shares the same density $|u_i|^2=|\exp(\ci \omega)u_i|^2$ (for $i=1,2$) and the same energy level $E(\bfu) = E(\exp(\ci \omega)\bfu)$. Hence, for any $\omega\in [-\pi, \pi)$ we have that $\exp(\ci \omega)\bfu$ is another (though physically equivalent) ground state.
\end{theorem}
\subsubsection{First and second order conditions for minimizers}
\label{subsection:first-second-order-cond}
Classically, any global minimizer $\bfu \in \boldS$ of $E$ must fulfill the first order and second order optimality conditions, cf. \cite{PHMY24}. If $\lambda$ denotes the corresponding Lagrange multiplier given by $\lambda := \langle \mathcal{L}_{\bfu}\bfu, \bfu \rangle$, then the {\it first order condition} is just the GPE \eqref{eq:eigenvalue} and reads
\begin{eqnarray*}
\langle E^{\prime}(\bfu) , \bfv \rangle - \, \lambda \, ( \bfu , \bfv )_{\bfL^2(\D)} &=& 0
\qquad \mbox{for all } \bfv \in \bfH_0^1(\D),
\end{eqnarray*}
whereas the necessary {\it second order condition} reads
\begin{eqnarray*}
\langle E^{\prime\prime}(\bfu) \bfv , \bfv \rangle - \, \lambda \, ( \bfv , \bfv )_{\bfL^2(\D)} &\ge& 0
\qquad \mbox{for all } \bfv \in T_{\bfu}\boldS,
\end{eqnarray*}
where $T_{\bfu}\boldS$ denotes the tangent space at $\bfu$ given by
\begin{eqnarray*}
T_{\bfu}\boldS &=& 
\{ \, \bfv \in \bfH^1_0(\D) \,\, | \,\, ( \bfu, \bfv )_{\bfL^2(\D)} = 0 \, \}.
\end{eqnarray*}
Due to the invariance of the energy under complex phase shifts $\exp(\ci \omega)$, the following must hold
\begin{align}
\label{secE-iu-ev}
\langle E^{\prime\prime}(u)(\ci \bfu ) , \bfv \rangle - \, \lambda \, (\ci \bfu , \bfv )_{\bfL^2(\D)} \,\,\,=\,\,\, 0
\qquad \mbox{for all } \bfv \in T_{\bfu}\boldS.
\end{align}
This is seen by considering the energy curve $\omega \mapsto E(\exp(\ci \omega) \bfu )$, which is constant. Hence, the first and the second derivative of the curve are zero. Computing these derivatives just yields \eqref{secE-iu-ev}. Therefore, $\lambda$ is always the smallest eigenvalue of $E^{\prime\prime}(\bfu)$ with corresponding eigenfunction $\ci \bfu$. If all other eigenvalues of $E^{\prime\prime}(\bfu)$ are strictly larger than $\lambda$, then $\bfu$ is called {\it locally quasi-isolated} \cite{PHMY24,PHMY242}.

A straightforward calculation shows that the second derivative of $E$ is given by
\begin{eqnarray*}
\lefteqn{ \langle E''(\bfu) \bfv , \bfw \rangle }\\
&=& \langle \mathcal{L}_{\bfu} \bfv , \bfw \rangle +2\Re\int_{\D}\sum_{j = 1}^2\beta_{jj}\Re(u_j\bar{v}_j)u_j\conjw{j} + \beta_{12}\big(\Re(u_1\bar{v}_1)u_2\conjw{2} + \Re(u_2\bar{v}_2)u_1\conjw{1}\big)\dx.
\end{eqnarray*}
This can be used in practice to verify a posteriori whether a computed state $\bfu \in \boldS$ is indeed a local minimizer of $E$ by calculating the smallest eigenvalues of $E''(\bfu)\vert_{T_{\bfu}\boldS}$ which have to be in agreement with the above conditions. 
\subsubsection{Energy-adaptive metric}
Since we want to use the metric induced by the linear operator $\mathcal{L}_{\bfu}$, we need to verify that it is elliptic. For this, we note that $\mathcal{L}_{\bfu}$ can be written as 
\begin{align}
\label{Lu-as-Lzero}
	\hspace{-5pt}\langle \mathcal{L}_{\bfu}\bfv,\bfw\rangle \,=\, \langle \mathcal{L}_{\bfzero}\bfv ,\bfw \rangle +\sum_{j = 1}^2\beta_{jj} \Re\int_{\D}|u_j|^2 v_j\conjw{j}\dx + \beta_{12} \Re\int_{\D}|u_1|^2v_2\conjw{2} + |u_2|^2v_1\conjw{1} \dx,
\end{align}
where $\mathcal{L}_{\bfzero}$ denotes the $\bfu$-independent part of $\mathcal{L}_{\bfu}$. With the convention $\beta_{21}:=\beta_{12}$ we obtain  
\begin{eqnarray*}
\langle \mathcal{L}_{\bfu}\bfv,\bfv\rangle 
&=& \langle \mathcal{L}_{\bfzero}\bfv ,\bfv \rangle +\sum_{i,j = 1}^2\beta_{ij} \int_{\D}|u_i|^2 |v_j|^2 \dx 
\,\,\, \ge \,\,\, \langle \mathcal{L}_{\bfzero}\bfv ,\bfv \rangle.
\end{eqnarray*}
Hence, $\mathcal{L}_{\bfu}$ is elliptic for any $\bfu\in \bfH^1_0(\D)$ as long as $\mathcal{L}_{\bfzero}$ is elliptic. It is therefore sufficient to prove the following result.
\begin{lemma}
\label{lemma:ellipticity-L0}
Assume \ref{A1}-\ref{A2}, then $\mathcal{L}_{\bfzero}$ is coercive (elliptic), i.e., there is a constant $\alpha>0$ such that
\begin{align}
\label{ellipticity-L0-est}
\langle \mathcal{L}_{\bfzero} \bfv , \bfv \rangle \,\,\, \ge \,\,\, \frac{1}{4} 
\| \bfv \|_{\bfH^1(\D)}^2\qquad
\mbox{for all } \bfv \in \bfH^1_0(\D).
\end{align}
\end{lemma}
\begin{proof}
According to \eqref{def-Lu}, $\langle \mathcal{L}_{\bfzero} \bfv, \bfv \rangle$ is given by
\begin{eqnarray*}
 \langle \mathcal{L}_{\bfzero} \bfv, \bfv \rangle &=& 
	\int_{\D} \sum_{j = 1}^2\frac{1}{2} |\nabla v_j|^2 + V_j |v_j|^2
	+ \frac{\delta}{2}\big( |v_1|^2 - |v_2|^2\big) + \Omega \, \Re \big(v_1\conjv{2}\big) \\
	&\enspace& \,\, + \,\Re
	\int_{\D} \ci k_0\big(\conjv{1}\partial_1v_1 - \conjv{2}\partial_1 v_2\big)\dx. 
\end{eqnarray*}
To find a lower bound, we start with the last term, where we use Young's inequality for $\varepsilon>0$ to obtain
		\[|\Re\big(\ci k_0 \int_{\D} \bar{v}_1\partial_1v_1 - \bar{v}_2\partial_1v_2\big)|\dx\leq \frac{k_0}{2}\int_{\D}\frac{|v_1|^2 + |v_2|^2}{\varepsilon} + \varepsilon(|\nabla v_1|^2 + |\nabla v_2|^2)\dx.\]
With the choice \(\varepsilon^{-1} = 2k_0\) and by using the Cauchy-Schwarz inequality, this yields
\begin{eqnarray*}
\langle \mathcal{L}_{\bfzero} \bfv, \bfv \rangle \geq \frac{1}{4}\|\nabla \bfv\|_{\Ltwo}^2 + \sum_{j = 1}^2\int_{\D} \big(V_j(x) - \frac{|\delta| + |\Omega| + 2k_0^2}{2}\big)|v_j|^2 \dx.
\end{eqnarray*}
By assumption \ref{A2} the result follows.
\end{proof}
Recalling that $\mathcal{L}_{\bfu}$ is self-adjoint, we conclude from Lemma \ref{lemma:ellipticity-L0} that $\langle \mathcal{L}_{\bfu} \bfv , \bfw \rangle $ defines an inner product. We fix this in the following conclusion.
\begin{conclusion}
\label{conclusion-coercivity-Lu}
Assume \ref{A1}-\ref{A3} and let $\bfu \in \bfH^1_0(\D)$ be arbitrary. Then $\langle \mathcal{L}_{\bfu} \bfv , \bfw \rangle $ is symmetric and it holds
\begin{align*}
\langle \mathcal{L}_{\bfu} \bfv , \bfv \rangle \,\,\, \ge \,\,\, \frac{1}{4} 
\| \bfv \|_{\bfH^1(\D)}^2\qquad
\mbox{for all } \bfv \in \bfH^1_0(\D).
\end{align*}
Hence, $\langle \mathcal{L}_{\bfu} \bfv , \bfw \rangle $ is an inner product on $\bfH^1_0(\D)$.
\end{conclusion}
Now that we verified that $\langle \mathcal{L}_{\bfu} \bfv , \bfv \rangle$ defines an inner product that is linked to the energy, we will use it in the next step as an adaptively changing metric in a Riemannian gradient method analogously to what we did in Section \ref{subsection:metric-steepest-descent-GPE}.

\subsection{Metric-driven steepest descents for SO-coupled BECs}
Since the arguments that we developed in Section \ref{subsection:metric-steepest-descent-GPE} are generally applicable, the metric-driven steepest descent in the adaptive $\langle \mathcal{L}_{\bfu} \bfv , \bfw \rangle$-metric for computing minimizers of \eqref{energy-minimization-problem-SO-coupled-BEC} is given as follows.
\begin{definition}[Metric-driven Riemannian gradient method for SO-coupled BECs]
\label{definition-metric-driven-grad-method-SO-coupl-BEC}
For a starting value $\bfu_0 \in \boldS$ and a sequence of (pseudo) time step sizes $\tau_n>0$, the iterations are given (for $n\ge 0$) by 
\begin{eqnarray}
\label{RSG-SOcoupl-BEC}
\bfu^{n+1}(\tau) &:=& \frac{ \hspace{-23pt}(1 - \tau)\,\bfu^n  \,+\, \tau \, ( \mathcal{L}_{\bfu^n}^{-1} \bfu^n , \bfu^n )_{\bfL^2(\D)}^{-1} \,\mathcal{L}_{\bfu^n}^{-1} \bfu^n }{\| (1 - \tau)\,\bfu^n  \,+\, \tau \, ( \mathcal{L}_{\bfu^n}^{-1} \bfu^n , \bfu^n )_{\bfL^2(\D)}^{-1} \,\mathcal{L}_{\bfu^n}^{-1} \bfu^n \|_{\bfL^2(\D)} } 
\end{eqnarray}
and we set $\bfu^{n+1}=\bfu^{n+1}(\tau^n)$ for the optimal step size
\begin{eqnarray}
\label{opt-step-size-SO-BEC}
\tau_n &=&
\underset{0<\tau \le 2}{\mbox{\normalfont arg\hspace{1pt}min}} \,\, E\hspace{-1pt}\left( \bfu^{n+1}(\tau) \right).
\end{eqnarray} 
Note that each $\bfu^n \in \bfH^1_0(\D)$ consists of two components $\bfu^n=(u^n_1,u^n_2)$.
\end{definition}
Since the application of the metric-driven approach to SO-coupled BECs is new, it remains to verify that the method is indeed strictly energy-diminishing (for all sufficiently small step sizes $\tau_n$, and in particular the optimal step size) and that the method converges, for any starting value, to a critical point of the energy. For that, we transfer the techniques previously developed in e.g. \cite{HeP20,PHMY242} to the new setting to obtain the following convergence result.
\begin{theorem}\label{theorem:energy_diss-SO-BEC}
Assume \ref{A1}-\ref{A3}. Consider the iterations \eqref{RSG-SOcoupl-BEC} for some given (not necessarily optimal) step sizes $\tau_n$ and a starting value $\bfu_0 \in \boldS$. Then there exists a step size interval $[\tau_{\mins}, \tau_{\maxs}] \subset (0,2)$ such that if $\tau_n \in [\tau_{\mins}, \tau_{\maxs}]$ for all $n$, then the iterates $\bfu^n \in \bfH^1_0(\D)$ in \eqref{RSG-SOcoupl-BEC} have the following properties:
  \begin{enumerate}[label={(\roman*)}]
   \item\label{enum:1} The energy is diminished in each iteration. In particular, there is a \(C_\tau>0\) such that 
   $E(\bfu^n)- E(\bfu^{n+1})\geq C_\tau \| \bfu^{n+1} - \bfu^n\|_{\Hone}^2$.
   \item\label{enum:2} The energy decays to a limit energy, i.e., $\lim\limits_{n\to \infty}E(\bfu^n) = E_0\in \mathbb{R}_{\geq 0}$.
   \item \label{enum:3}The sequence of iterates possesses a subsequence \( \{ \bfu^{n_j}\}_{j\in \mathbb{N}}\) that converges strongly to a critical point $\bfu\in \boldS$ of $E$, i.e., it holds $\lim\limits_{j\to\infty} \| \bfu^{n_j} - \bfu\|_{\Hone} = 0$ and $\bfu\in \boldS$ fulfills the equation $E^{\prime}(\bfu) = \lambda\, (\bfu , \cdot)_{\bfL^2(\D)}$ for the eigenvalue $\lambda :=  \langle E^{\prime}({\bfu}) , \bfu \rangle = \lim\limits_{j\rightarrow \infty} \langle\mathcal{L}^{-1}_{\bfu^{n_j}} \bfu^{n_j}, \bfu^{n_j} )_{\Ltwo}^{-1}$.
   \item \label{enum:4}
   If the limit of the subsequence in \ref{enum:3} is a locally quasi-isolated ground state in the sense of Section \ref{subsection:first-second-order-cond}, then the entire sequence of density iterates \(|\bfu^n|^2\) converges to the corresponding ground state density \(|\bfu|^2\), i.e.
   $$
   \lim\limits_{n\to \infty}\| \, |\bfu^n|^2 - |\bfu|^2 \, \|_{\Ltwo} = 0. 
   $$
  \end{enumerate}
\end{theorem}
The proof of the theorem is given in Section \ref{subsection-convergence-proof-SO-BEC}. Before presenting it, let us make some final remarks. First, we note that Theorem \ref{theorem:energy_diss-SO-BEC} covers in particular the metric-driven Riemannian gradient method with optimal step size according to \eqref{opt-step-size-SO-BEC}. Second, note that in \ref{enum:3}, we can only guarantee convergence up to subsequences due to the energy functional being invariant under complex phase shifts. The iteration does not take this into account, hence the iteration itself might continuously change the phase, whereas the density $|\bfu^n|$ remains convergent. This is because the modulus itself is unaffected by phase shifts.

\subsection{Convergence proof}
\label{subsection-convergence-proof-SO-BEC}
In this section we give a proof of Theorem \ref{theorem:energy_diss-SO-BEC} using the same techniques as developed in \cite{HeP20,PHMY242} for single-component BECs. To keep the notation compact, we define, for a given previous iterate $\bfu^n$ and a step size $\tau_n$, the preliminary iterate as
\begin{eqnarray}
\label{RSG-SOcoupl-BEC-proof}
\bfuprel{n+1} &:=&  (1 - \tau_n)\,\bfu^n  \,+\, \tau_n \,  \gamma_{\bfu^n} \,\mathcal{L}_{\bfu^n}^{-1} \bfu^n, 
\qquad \mbox{where } \gamma_{\bfu^n} := ( \mathcal{L}_{\bfu^n}^{-1} \bfu^n , \bfu^n )_{\bfL^2(\D)}^{-1}.
\end{eqnarray}
We also denote $\bfuprel{n+1}=(\uprel{1}{n+1},\uprel{2}{n+1})$ and
the new iterate is $\bfu^{n+1}:=\bfuprel{n+1}/\| \bfuprel{n+1} \|_{L^2(\D)}$. 
Furthermore, the energy norm induced by $\mathcal{L}_{\bfzero}$ shall be denoted by
$$
\energynorm{\bfv} = \langle \mathcal{L}_{\bfzero} \bfv,\bfv\rangle^{1/2}.
$$
We now turn to the proof of Theorem \ref{theorem:energy_diss-SO-BEC}, which requires a series of lemmas which will be proved first. All these auxiliary results will silently assume that we are in the setting of Theorem \ref{theorem:energy_diss-SO-BEC}.

We start by noting that the preliminary iterations increase the mass.

\begin{lemma}\label{lemma:normprelim}
    If $\bfuprel{n+1} \not= \bfu^n$,
    it holds
	\[ \| \bfuprel{n+1}\|_{\Ltwo} - \| \bfu^n\|_{\Ltwo} > 0.\]
	In particular, since \(\bfu^n\in\boldS\), we have $\| \bfuprel{n+1}\|_{L^2(\D)} \ge 1$ for all $n$.
	\begin{proof}
		Denote for brevity the Riemannian gradient by $\bfd^n := -\bfu^n +\gamma_{\bfu^n} \,\mathcal{L}_{\bfu^n}^{-1} \bfu^n$ and recall $\gamma_{\bfu^n} := ( \mathcal{L}_{\bfu^n}^{-1} \bfu^n , \bfu^n )_{\bfL^2(\D)}^{-1}$. This yields
        \begin{align}
        \label{L2-orth-dn-un}
        (\bfu^n, \bfd^n)_{\Ltwo} = - \|\bfu^n\|_{\Ltwo}^2 + \gamma_{\bfu^n} (\bfu^n, \mathcal{L}_{\bfu^n}^{-1}\bfu^n)_{\Ltwo} = 0
        \end{align}
		and we conclude
		\begin{align*}
			\|\bfuprel{n+1}\|_{\Ltwo}^2&\,\,=\,\,\| \bfu^n\|_{\Ltwo}^2+2\tau_n(\bfu^n, \bfd^n)_{\Ltwo}+\tau_n^2\|\bfd^n\|_{\Ltwo}^2 \\
			&\,\,=\,\,\| \bfu^n \|_{\Ltwo}^2+\tau_n^2\| \bfuprel{n+1} - \bfu^n\|_{\Ltwo}^2
            \,\, > \,\, \| \bfu^n \|_{\Ltwo}^2.
		\end{align*}
	\end{proof}
\end{lemma}
Note that \(\| \bfuprel{n+1}\|_{L^2(\D)} \geq1\) implies $E(\bfu^{n+1})\le E(\bfuprel{n+1} )$. Hence, it is sufficient to show $ E(\bfuprel{n+1})\le E(\bfu^n)$ to conclude energy dissipation $E(\bfu^{n+1} ) \le E(\bfu^n)$. This is established in the following lemma with a quantitative lower bound.

\begin{lemma}\label{lemma:prelim}
    For \(\tau_n\leq\frac{1}{2}\) it holds
    \begin{eqnarray*}
		\lefteqn{ E(\bfu^n)-E(\bfuprel{n+1}) }\\
		&\geq &\big(\tfrac{1}{\tau_n} - \tfrac{1}{2}\big) \energynorm{ \bfuprel{n+1}-\bfu^n }^2
		-\int_{\D}\sum_{j=1}^2\tfrac{3\beta_{jj}}{4}
		|\uprel{j}{n+1} - u^n_j|^4
		+ \tfrac{3\beta_{12}}{2}
		|\uprel{1}{n+1}-u^n_1|^2|\uprel{2}{n+1} - u^n_2|^2\dx.
	\end{eqnarray*}
\end{lemma}

\begin{proof}
For $\bfv=(v_1,v_2)\in \bfH^1_0(\D)$ we 
express the energy $E$ in \eqref{eq:SO-energy} through the operator $\mathcal{L}_{\bfv}$ as
\begin{align}
\label{energy-in-terms-Lv}
    E(\bfv) = \tfrac{1}{2}\langle \mathcal{L}_{\bfv}\bfv,\bfv\rangle - \int_{\D}\tfrac{\beta_{11}}{4}|v_1|^4 + \tfrac{\beta_{22}}{4}|v_2|^4 + \tfrac{\beta_{12}}{2}|v_1|^2|v_2|^2\dx.
\end{align}
	This yields
	\begin{eqnarray}
    \label{eq:prelim}
\lefteqn{ E(\bfu^n)-E(\bfuprel{n+1})
		\,\,\,=\,\,\,\tfrac{1}{2}\Big(
			\langle \mathcal{L}_{\bfu^n}\bfu^n, \bfu^n\rangle
            - \langle \mathcal{L}_{\bfuprel{n}} \bfuprel{n} , \bfuprel{n}\rangle
		\Big) }\\
		&\enspace& + \int_{\D}\sum_{j=1}^2\tfrac{\beta_{jj}}{4} \left(|\uprel{j}{n+1}|^4 -|u^n_j|^4\right)
			\,+\,\tfrac{\beta_{12}}{2}\left(|\uprel{1}{n+1}|^2|\uprel{2}{n+1}|^2-|u^n_1|^2|u^n_2|^2 \right)\dx.\notag
	\end{eqnarray}		
    To rewrite the first term in this identity, we first note that
	\begin{align}
    \label{orth-Lun-error}
		\tfrac{1}{\tau_n}\langle 
			\mathcal{L}_{\bfu^n}(\bfuprel{n+1} - \bfu^n),
			\bfuprel{n+1} - \bfu^n\rangle \,\,\overset{\eqref{L2-orth-dn-un}}{=}\,\, - \langle \mathcal{L}_{\bfu^n}\bfu^n,	\bfuprel{n+1} - \bfu^n\rangle.
	\end{align}
    Together with the self-adjointness of \(\mathcal{L}_{\bfu^n}\), this implies
		\begin{eqnarray*}
		\lefteqn{ \langle \mathcal{L}_{\bfu^n}(\bfuprel{n+1} - \bfu^n),
		\bfuprel{n+1} - \bfu^n\rangle  \,\,\,=\,\,\,
		\langle \mathcal{L}_{\bfu^n}\bfuprel{n+1}, \bfuprel{n+1}\rangle
		+ \langle \mathcal{L}_{\bfu^n} \bfu^n, \bfu^n\rangle
		-2\langle \mathcal{L}_{\bfu^n} \bfu^n, \bfuprel{n+1}\rangle }\\
		&=&	\langle \mathcal{L}_{\bfu^n}\bfuprel{n+1}, \bfuprel{n+1}\rangle
		- \langle \mathcal{L}_{\bfu^n} \bfu^n, \bfu^n\rangle
		+ 2\langle \mathcal{L}_{\bfu^n} \bfu^n, \bfu^n \rangle
		- 2\langle \mathcal{L}_{\bfu^n} \bfu^n,\bfuprel{n+1} \rangle
		\rangle
		\\
		&=&	\langle \mathcal{L}_{\bfu^n}\bfuprel{n+1}, \bfuprel{n+1}\rangle
		- \langle \mathcal{L}_{\bfu^n} \bfu^n, \bfu^n\rangle
		+\tfrac{2}{\tau_n}\langle \mathcal{L}_{\bfu^n} (\bfuprel{n+1}-\bfu^n), \bfuprel{n+1} - \bfu^n\rangle.\hspace{80pt}
	\end{eqnarray*}
We obtain
	\begin{equation*}
		\langle \mathcal{L}_{\bfu^n}\bfu^n, \bfu^n\rangle \,\,=\,\, \langle \mathcal{L}_{\bfu^n}\bfuprel{n+1}, \bfuprel{n+1}\rangle + \big(\tfrac{2}{\tau_n} - 1\big)\langle \mathcal{L}_{\bfu^n}(\bfuprel{n+1} - \bfu^n),
		 \bfuprel{n+1} - \bfu^n\rangle.
	\end{equation*}
	We can now plug this result in equation \eqref{eq:prelim}. Noting that the terms \(\langle \mathcal{L}_{\bfu^n}\bfuprel{n+1}, \bfuprel{n+1}\rangle\) and \(\langle \mathcal{L}_{|\bfuprel{n+1}|}\bfuprel{n+1}, \bfuprel{n+1}\rangle\) only differ in their non-linear parts and hence the linear parts cancel each other out. This yields
		\begin{align*}
			E(\bfu^n) - E(\bfuprel{n+1})  = \,&\big(\tfrac{1}{\tau_n} - \tfrac{1}{2}\big)\langle
				\mathcal{L}_{\bfu^n}(\bfuprel{n+1} - \bfu^n),
				\bfuprel{n+1} - \bfu^n
			 \rangle - \sum_{j=1}^2 \tfrac{\beta_{jj}}{4}\int_{\D} (|\uprel{j}{n+1}|^2 - |u^n_j|^2)^2\dx\\
			 &-\tfrac{\beta_{12}}{2}\int_{\D}\big(|u^n_1|^2 - |\uprel{1}{n+1}|^2\big)
			 \big(|u^n_2|^2 - |\uprel{2}{n+1}|^2\big)\dx.
		\end{align*}
        Recalling the notation $\energynorm{\bfv}^2 = \langle \mathcal{L}_{\bfzero} \bfv,\bfv\rangle$ it follows
        \begin{align*}
            &E(\bfu^n) - E(\bfuprel{n+1}) \,\,\, = \,\,\, \big(\tfrac{1}{\tau_n} - \tfrac{1}{2}\big)
			\energynorm{ \bfuprel{n+1} - u^n }^2  \\
			&\,+\,\big(\tfrac{1}{\tau_n} - \tfrac{1}{2}\big)\Big(
			\sum_{j=1}^2\beta_{jj}\hspace{-3pt}\int_{\D}|u^n_j|^2|
			\uprel{j}{n+1} - u^n_j
			|^2\dx
			+\beta_{12}\int_{\D}|u^n_1|^2|\uprel{2}{n+1} - u^n_2|^2 + |u^n_2|^2|\uprel{1}{n+1} - u^n_1|^2\dx
			\Big)
			\\
			&\, - \int_{\D} \left( \sum_{j=1}^2 \tfrac{\beta_{jj}}{4} (|\uprel{j}{n+1}|^2 - |u^n_j|^2)^2
			+ \tfrac{\beta_{12}}{2} \big(|u^n_1|^2 - |\uprel{1}{n+1}|^2\big)
			\big(|u^n_2|^2 - |\uprel{2}{n+1}|^2\big) \right)\dx.
        \end{align*}
		A simple application of the triangle inequality and Young's inequality helps in the estimation of the last two terms, where we use
		\begin{eqnarray*}
			(|\uprel{j}{n+1}|^2 - |u^n_j|^2)^2&\leq& \,3|\uprel{j}{n+1}- u^n_j|^4 + 6|\uprel{j}{n+1} - u^n_j|^2|u^n_j|^2 \qquad \mbox{for } j=1,2
        \end{eqnarray*}
        and
		\begin{eqnarray*}
		\lefteqn{ (|u^n_1|^2 - |\uprel{1}{n+1}|^2)(|u^n_2|^2 - |\uprel{2}{n+1}|^2) }\\
            &\leq& \, 3 \left(
			|\uprel{1}{n+1}-u^n_1|^2|\uprel{2}{n+1}- u^n_2|^2
			+ |\uprel{1}{n+1} - u^n_1|^2|u^n_2|^2
			+ |\uprel{2}{n+1} - u^n_2|^2|u^n_1|^2
			\right).
		\end{eqnarray*}
		We conclude that 
            \begin{align*}
			 &E(\bfu^n)-E(\bfuprel{n+1}) \\
			&\,\geq \big(\tfrac{1}{\tau_n} - \tfrac{1}{2}\big) \energynorm{ \bfuprel{n+1}-u^n }^2 -\tfrac{3}{2}\beta_{12}\int_{\D}
				|\uprel{1}{n+1}-u^n_1|^2
				|\uprel{2}{n+1} - u^n_2|^2\dx
			-\tfrac{3}{4}\sum_{j=1}^2\beta_{jj}\int_{\D}
				|\uprel{j}{n+1} - u^n_j|^4\dx
			\\
			&\, +\,\big(\tfrac{1}{\tau_n}-2\big)
		\Big(
			\sum_{j=1}^2 \beta_{jj}\int_{\D}|u^n_j|^2 |\uprel{j}{n+1} - u^n_j|^2\dx
			+\beta_{12}\int_{\D} 
			|u^n_1|^2 |\uprel{2}{n+1} - u^n_2|^2
			+ |u^n_2|^2 |\uprel{1}{n+1} - u^n_1|^2\dx
		\Big).
            \end{align*}
			If we pick the step size \(\tau_n\) such that \(\tau_n\leq \tfrac{1}{2}\), this implies \(\tfrac{1}{\tau_n} - 2 \geq 0\) and hence the result follows.
\end{proof}
With the previous lemma, we obtain the following lower bound on the energy reduction per iteration.
\begin{lemma}\label{lemma:Rnorm}
    There exists some \(\tau_\text{max}<2\), depending only on the non-linear interaction parameters, \(\D\) and the initial energy \(E(\bfu^0)\), such that for all \(\tau_n\in (0, \tau_{\text{max}})\) it holds
	\begin{align}
    \label{estEun-Euprel}
    E(\bfu^n) - E(\bfuprel{n+1}) \,\,\geq\,\, C_{\tau_n} \, \energynorm{ \bfu^n - \bfuprel{n+1} }^2.
    \end{align}
    The constant fulfills $C_{\tau_n} \ge C_{\tau_{\text{max}}} >0$ for all $n$. Furthermore, it holds
    \begin{align}
    \label{L2-norm-conv-prelun}
    \lim_{n\rightarrow \infty} \| \bfuprel{n} \|_{L^2(\D)} = 1.
    \end{align}
	\begin{proof}
First, note that
\begin{align}
\label{energy-in-terms-Lv-est}
    2\, E(\bfv) \,\,\,\overset{\eqref{eq:SO-energy}}{\ge}\,\,\, E(\bfv) \, + \, \int_{\D}\tfrac{\beta_{11}}{4}|v_1|^4 + \tfrac{\beta_{22}}{4}|v_2|^4 + \tfrac{\beta_{12}}{2}|v_1|^2|v_2|^2 \dx
    \,\,\,\overset{\eqref{energy-in-terms-Lv}}{=}\,\,\, \tfrac{1}{2}\langle \mathcal{L}_{\bfv}\bfv,\bfv\rangle. 
\end{align}
We obtain
\begin{eqnarray*}
\lefteqn{ \energynorm{ \bfuprel{n+1}- \bfu^n }^2 
        \,\,\,=\,\,\, \langle\mathcal{L}_{\bfzero}(\bfuprel{n+1}-\bfu^n), \bfuprel{n+1}-\bfu^n\rangle 
	 	\,\,\,\leq\,\,\, \langle\mathcal{L}_{\bfu^n}(\bfuprel{n+1} - \bfu^n), \bfuprel{n+1} - \bfu^n \rangle }\\
	 	&\overset{\eqref{orth-Lun-error}}{=}& -\tau_n\langle \mathcal{L}_{\bfu^n}\bfu^n, \bfuprel{n+1} - \bfu^n \rangle 
	 	\,\,\,=\,\,\,-\tau_n\langle\mathcal{L}_{\bfu^n}\bfu^n, 
	 		-\tau_n \bfu^n + \tau_n \gamma_{\bfu^n} \mathcal{L}_{\bfu^n}^{-1} \bfu^n
	 	\rangle\\
	 	&=&\tau_n^2 \, \langle\mathcal{L}_{\bfu^n}\bfu^n, \bfu^n \rangle
        -\tau_n^2 \, \gamma_{\bfu^n} \, \langle\mathcal{L}_{\bfu^n}\bfu^n, 
	     \mathcal{L}_{\bfu^n}^{-1} \bfu^n
	 	\rangle
        \,\,\overset{\gamma_{\bfu^n} \ge 0}{\le}\,\, \tau_n^2 \, \langle\mathcal{L}_{\bfu^n}\bfu^n, \bfu^n \rangle
        \,\,\overset{\eqref{energy-in-terms-Lv-est}}{\le}\,\,
        4\, \tau_n^2 \,E(\bfu^n).
\end{eqnarray*}
Hence, for all $n\ge0$ 
	 \begin{equation}\label{eq:normenergy}
	 \energynorm{ \bfuprel{n+1} - \bfu^n }^2 \,\,\,\leq\,\,\, 
     4 \, \tau_n^2 \,E(\bfu^n).
	 \end{equation}
	 The statement \eqref{estEun-Euprel} will now follow by induction: Let \(n=0\), then inequality \eqref{eq:normenergy} implies
	 \begin{equation}
	 	\energynorm{ \bfuprel{1} - \bfu^0 }^2 \,\,\,\leq\,\,\, 4\,\tau_0^2\, E(\bfu^0) <1, \label{eq:Rnorm}
	 \end{equation}
	 if \(\tau_0\) is chosen such that \(\tau_0^2 \, < \, (4\, E(\bfu^0))^{-1}\). 
     From Lemma \ref{lemma:prelim} and Young's inequality it follows	 
    \begin{align*}
        &E(\bfu^0)-E(\bfuprel{1}) \\
	 	&\,\geq \big(\tfrac{1}{\tau_0}-\tfrac{1}{2}\big)
	 	\energynorm{ \bfuprel{1}- \bfu^0 }^2 -\tfrac{3}{4}\sum_{j=1}^2\beta_{jj} \int_{\D}
	 		|\uprel{j}{1}-u^0_j|^4\dx
	 	-\tfrac{3}{2}\beta_{12} \int_{\D}
	 	|\uprel{1}{1}- u^0_1|^2|\uprel{2}{1} - u^0_2|^2\dx\\
	 	&\,\geq 	\big(\tfrac{1}{\tau_0}-\tfrac{1}{2}\big) \energynorm{ \bfuprel{1}-\bfu^0 }^2-\tfrac{3}{4} \, C_{\boldsymbol{\beta}} \, \|\bfuprel{1}-\bfu^0\|_{\boldsymbol{L}^4(\D)}^4
    \end{align*}
     for some constant $C_{\boldsymbol{\beta}}>0$ that depends on $\beta_{11}$, $\beta_{12}$ and $\beta_{22}$. 
	 Due to the norm equivalence of \(\|\cdot\|_{\Hone}\) and $\energynorm{\cdot }$ (guaranteed by Lemma \ref{lemma:ellipticity-L0}) as well as the Sobolev embedding theorem, there exists a constant \(C_{\mbox{\tiny Sob}}>0\) such that
	 \[\|\bfuprel{1} - \bfu^0\|^4_{L^4(\D)^2} \leq C_{\mbox{\tiny Sob}} \|\bfuprel{1} - \bfu^0\|^4_{\Hone} \overset{\eqref{ellipticity-L0-est}}{\leq} 16 \, C_{\mbox{\tiny Sob}} \energynorm{ \bfuprel{1} - \bfu^0 }^4\overset{\eqref{eq:Rnorm}}{\leq}  16\, C_{\mbox{\tiny Sob}} \energynorm{ \bfuprel{1} - \bfu^0 }^2.\]
	 Denote $\tilde{C}_{\boldsymbol{\beta}}:=12 C_{\mbox{\tiny Sob}} C_{\boldsymbol{\beta}}$, then a combination of the previous two estimates results in
	 \begin{align*}
	 E(\bfu^0)-E(\bfuprel{1}) 
	 &\,\,\geq\,\, \big(\tfrac{1}{\tau_0}-\tfrac{1}{2}- \tilde{C}_{\boldsymbol{\beta}} \big) \energynorm{ \bfuprel{1} - \bfu^0 }^2 \,\,=:\,\, C_{\tau_0} \,\energynorm{ \bfuprel{1} - \bfu^0 }^2.
	 \end{align*}
	 The constant \(C_{\tau_0}\) is strictly positive as long as \(\tau_0 \le \tau_{\text{max}}\) for a sufficiently small $\tau_{\text{max}}$. By Lemma \ref{lemma:normprelim}, we also have \(E(\bfu^1)\leq E(\bfu^0)\).\\[0.3em]
	 For the induction step, assume \(E(\bfu^n)\leq E(\bfu^0)\) and let \(\tau_{\text{max}}>0\) be sufficiently small for step 1. Then for all \(\tau_n\leq\tau_{\text{max}}\) and with equation \eqref{eq:normenergy} we can still verify 
	 $\energynorm{ \bfuprel{n+1} - \bfu^n }^2\leq 4\tau_n^2E(\bfu^n)\leq 4\tau_n^2 E(\bfu^0)< 1$. Consequently, we can argue as before with the same constants to verify for $C_{\tau_n} = \tfrac{1}{\tau_n}-\tfrac{1}{2}- \tilde{C}_{\boldsymbol{\beta}}$ that
	 \,\,$E(\bfu^n) - E(\bfuprel{n+1}) \,\,\geq\,\, C_{\tau_n} \, \energynorm{ \bfuprel{n+1} - \bfu^n }^2$.

     For property \eqref{L2-norm-conv-prelun}, recall \(\| \bfuprel{n+1}\|_{L^2(\D)} \geq1\) and $E(\bfu^{n+1})\le E(\bfuprel{n+1} )$. The Poincar\'e inequality implies $\| \bfuprel{n+1} \|_{L^2(\D)} - 1 \le \| \bfuprel{n+1} - \bfu^n \|_{L^2(\D)} \lesssim \energynorm{ \bfuprel{n+1} - \bfu^n }$. Combining this with \eqref{estEun-Euprel} and exploiting that $\lim\limits_{n\rightarrow \infty} E(\bfu^n) - E(\bfuprel{n+1}) \le \lim\limits_{n\rightarrow \infty} E(\bfu^{n}) - E(\bfu^{n+1}) = 0$ (due to $E(\bfu^n)$ being a monotonically decreasing positive sequence), statement \eqref{L2-norm-conv-prelun} follows.
	\end{proof}
\end{lemma}
The lemmas above can be used to prove Theorem \ref{theorem:energy_diss-SO-BEC}.\ref{enum:1}.
\begin{corollary}[Theorem \ref{theorem:energy_diss-SO-BEC}.\ref{enum:1}.]\label{corollary:global}
        There exists \(\tau_{\text{max}}<2\) depending on $E(\bfu^0)$, $\D$, $\beta_{11}$, $\beta_{22}$ and $\beta_{12}$ such that for all \(\tau_n\in(0,\tau_{\text{max}})\) and a constant $C_{\tau_n} \ge C_{\tau_{\text{max}}} >0$ 
		\[E(\bfu^n) - E(\bfu^{n+1})\,\,\geq \,\,C_{\tau_n}\,\| \bfu^{n+1} - \bfu^n\|_{\Hone}^2.\]
		\begin{proof}
			The result follows from Lemmas \ref{lemma:ellipticity-L0}, \ref{lemma:normprelim} and \ref{lemma:Rnorm} if we can show \(\|\bfu^{n+1} - \bfu^n\|_{\Hone}\leq C \| \bfuprel{n+1} - \bfu^n \|_{\Hone} \) for some constant \(C>0\). 
            As a first step, we show that \(\|u^n\|_{\Hone}\) can be bounded in terms of the initial energy, where we get
				\begin{eqnarray}
                \label{uniform-H1-bound}
				\tfrac{1}{4} \|\bfu^n\|_{\Hone}^2 
                &\overset{\eqref{ellipticity-L0-est}}{\leq}& \energynorm{ \bfu^n }^2 
                \,\,\,\le\,\,\, \langle\mathcal{L}_{\bfu^n} \bfu^n,\bfu^n\rangle
                \,\,\,\overset{\eqref{energy-in-terms-Lv-est}}{\le}\,\,\, 4 \, E(\bfu^n) \,\,\,\le\,\,\, 4\, E(\bfu^0).
			\end{eqnarray}
Hence, $\|\bfu^n\|_{\Hone}\leq \sqrt{8} \sqrt{E(\bfu^0)} =: C_{\bfu^0}$.
			The corollary now follows from the fact that \(\|\bfuprel{n+1}\|_{L^2(\D)}\geq1\) (as per Lemma \ref{lemma:normprelim}):
				\begin{eqnarray*}
				\| \bfu^{n+1} - \bfu^n\|_{\Hone}&\leq&\|\bfuprel{n+1}\|_{\Ltwo}\|\bfu^{n+1} - \bfu^n\|_{\Hone} \\
				&=&\|\bfuprel{n+1}-\bfu^n\|_{\Hone} +|1-\bfuprel{n+1}
					\|_{\Ltwo}|
					\|	\bfu^n
				\|_{H^1(\mathcal{D})}\\
				&=&\|\bfuprel{n+1}-\bfu^n\|_{\Hone} +\Big(\|\bfuprel{n+1}\|_{\Ltwo}-\|\bfu^n\|_{\Ltwo}\|\Big)\|	\bfu^n
				\|_{\Hone} \\
				&\leq& \|\bfuprel{n+1}-\bfu^n\|_{\Hone} +\|\bfuprel{n+1}-\bfu^n\|_{\Hone}\|\bfu^n \|_{\Hone}\\
				&\leq&\big(1+C_{\bfu^0}\big)\, \|\bfuprel{n+1}-\bfu^n\|_{\Hone}
				\,\,\,\overset{\eqref{ellipticity-L0-est}}{\leq} \,\,\, 2\,\big(1+C_{\bfu^0}\big)\, \energynorm{ \bfuprel{n+1}-\bfu^n }.
			\end{eqnarray*}
            Since \(\| \bfuprel{n+1}\|_{L^2(\D)} \geq1\) yields $E(\bfu^{n+1})\le E(\bfuprel{n+1} )$, Lemma \ref{lemma:Rnorm} finishes the proof.
		\end{proof}
\end{corollary}
We are now ready to prove the remaining properties stated in Theorem \ref{theorem:energy_diss-SO-BEC}.
\begin{proof}[Proof of Theorem \ref{theorem:energy_diss-SO-BEC}]
	With property \ref{enum:1} already established, the remaining properties follow with classical compactness arguments.\\
    \ref{enum:2}: Since the sequence \(\{E(\bfu^n)\}_{n \in \mathbb{N}}\) is monotonically decreasing and bounded from below, we have existence of a limit energy \(\lim_{n\to \infty}E(\bfu^n) = E_0 \) and property \ref{enum:2} follows.\\
	\ref{enum:3}: By \eqref{uniform-H1-bound} we know that \(\{\bfu^n\}_{n \in \mathbb{N}}\) is uniformly bounded in $\bfH^1_0(\D)$. Due to the compact embedding of $H^1(\D)$ into $L^4(\D)$ (for $d\le 3$), we may extract a subsequence \(\{\bfu^{n_k}\}_{k\in\mathbb{N}}\) that converges to some \(\bfu\in \mathbb{S}\) weakly in \(\Hone\) and strongly in \(\boldsymbol{L^4}(\D)\). Using \eqref{L2-norm-conv-prelun}, we also have
	\[1=\lim_{n\to\infty}\|\bfuprel{n+1}\|_{\Ltwo}=\lim_{n\to\infty}\|
		(1-\tau_n)\bfu^n+\tau_n\gamma_{\bfu^n}\mathcal{L}_{\bfu^n}^{-1}\bfu^n
	\|_{\Ltwo}.\]
	By the coercivity of \(\langle \mathcal{L}_{\bfu}\,\cdot, \cdot\rangle\) in Conclusion \ref{conclusion-coercivity-Lu} 
    we also have
\begin{align}
\label{Lbfuinv-bfunk-con}
\|\mathcal{L}_{\bfu}^{-1}(\bfu^{n_k}-\bfu)\|_{\Hone} 
\,\,\, \lesssim \,\,\, \|\bfu^{n_k}-\bfu\|_{\Ltwo} \,\, \longrightarrow \,\, 0,
	\end{align}
    i.e., \(\mathcal{L}_{\bfu}^{-1}\bfu^{n_j}\longrightarrow \mathcal{L}_{\bfu}^{-1}\bfu\) in \(\Hone\). 
    To verify convergence of \(\mathcal{L}_{|\bfu^{n_j}|}^{-1}\bfu^{n_j}\) to \(\mathcal{L}_{\bfu}^{-1}\bfu\) (weakly in \(\Hone \) and strongly in \(\Ltwo\)), 
    recall that \(\mathcal{L}_{\bfu}\) and \(\mathcal{L}_{\bfu^{n_j}}\) only differ from each other in their non-linear parts, which for arbitrary \(\bfw\in \Hone\) yields
\begin{eqnarray*}
\lefteqn{ | \langle \mathcal{L}_{\bfu}
\big( \mathcal{L}_{\bfu}^{-1}\bfu^{n_k} -\mathcal{L}_{\bfu^{n_k}}^{-1}\bfu^{n_k} \big)
,\bfw \rangle | 
\,\,\,
=\,\,\ | \langle \mathcal{L}_{\bfu^{n_k}} \mathcal{L}_{\bfu^{n_k}}^{-1}\bfu^{n_k} 
- \mathcal{L}_{\bfu} \mathcal{L}_{\bfu^{n_k}}^{-1}\bfu^{n_k}
,\bfw \rangle | }\\
&\overset{\eqref{Lu-as-Lzero}}{=}&
\left| \Re \int_{\D} \sum_{j=1}^2 \left( \beta_{jj}
			\big(\mathcal{L}_{\bfu^{n_k}}^{-1}\bfu^{n_k}\big)_j\big(
			|u_j^{n_k}|^2-|u_j|^2
			\big)\conjw{j} \right)
		+\beta_{12}
		\big(
		\mathcal{L}_{\bfu^{n_k}}^{-1}\bfu^{n_k}\big)_1\big(|u_2^{n_k}|^2-|u_2|^2
			\big)\conjw{1} \right.\\
		&\enspace&\qquad\left. + \, \Re \int_{\D} \beta_{12} \big(
		\mathcal{L}_{\bfu^{n_k}}^{-1}\bfu^{n_k}\big)_2\big(|u_1^{n_k}|^2-|u_1|^2
			\big)\conjw{2} \dx\right|\\
&\lesssim& \| \,|\bfu^{n_k}|^2-|\bfu|^2 \|_{\Ltwo} \,
		\| \mathcal{L}_{\bfu^{n_k}}^{-1}\bfu^{n_k} \|_{\boldsymbol{L}^4(\D)} \| \bfw \|_{\boldsymbol{L}^4(\D)} \\
&\lesssim& \| \,|\bfu^{n_k}|^2-|\bfu|^2 \|_{\Ltwo} \,
		\| \bfu^{n_k} \|_{\boldsymbol{L}^2(\D)} \| \bfw \|_{\boldsymbol{H}^1(\D)} 
        \,\,\,=\,\,\, \| \,|\bfu^{n_k}|^2-|\bfu|^2 \|_{\Ltwo} \,
		\| \bfw \|_{\boldsymbol{H}^1(\D)} \\
&\lesssim&\| \bfu^{n_k}- \bfu \|_{\mathbf{L}^4(\D)} \,
		\| \bfw \|_{\boldsymbol{H}^1(\D)} \,\, \rightarrow \,\, 0
        \qquad \mbox{for } n_k\rightarrow \infty,
\end{eqnarray*}
where we used in the last two steps that $\bfu^{n_k}$ is uniformly bounded in $\bfH^1(\D)$ (and hence also in $\bfL^4(\D)$) such that $\bfu^{n_k}$ converges to $\bfu$ strongly in $\bfL^4(\D)$.
With the test function $\bfw=\mathcal{L}_{\bfu}^{-1}\bfu^{n_k} -\mathcal{L}_{\bfu^{n_k}}^{-1}\bfu^{n_k}$ and the coercivity of $\mathcal{L}_{\bfu}$ we therefore obtain
\begin{align*}
\| \mathcal{L}_{\bfu}^{-1}\bfu^{n_k} -\mathcal{L}_{\bfu^{n_k}}^{-1}\bfu^{n_k} \|_{\bfH^1(\D)} 
\,\,\lesssim\,\, \| \bfu^{n_k}- \bfu \|_{\mathbf{L}^4(\D)} \,\, \rightarrow \,\, 0
        \qquad \mbox{for } n_k\rightarrow \infty.
\end{align*}
Together with \eqref{Lbfuinv-bfunk-con}, we conclude that 
\begin{align*}
\mathcal{L}^{-1}_{\bfu^{n_k}}\bfu^{n_k} \longrightarrow \mathcal{L}^{-1}_{\bfu}\bfu \,\quad \text{ in } \, \Hone
\end{align*}
and hence also
\begin{align*}
\gamma_{\bfu^{n_k}} \,=\,(\bfu^{n_k}, \mathcal{L}^{-1}_{\bfu^{n_k}}\bfu^{n_k})_{\Ltwo}^{-1}\,\longrightarrow \,(\bfu,\mathcal{L}^{-1}_{\bfu} \bfu)_{\Ltwo}^{-1}\,=:\,\lambda.
\end{align*}
Since $\| \bfuprel{(n_k+1)} - \bfu^{n_k} \|_{\Hone}\rightarrow0$ (as a consequence of Lemma \ref{lemma:Rnorm}) and since $\tau_n$ is bounded from below and from above, we can pass to the limit in \eqref{RSG-SOcoupl-BEC-proof} to see that
\begin{eqnarray*}
\bfu^{n_k} =  - \, \tfrac{1}{\tau_{n_k}}(\bfuprel{n_k+1}-\bfu^{n_k}) \,+\, \,  \gamma_{\bfu^{n_k}} \,\mathcal{L}_{\bfu^{n_k}}^{-1} \bfu^{n_k} 
\,\longrightarrow  \, \lambda \, \mathcal{L}^{-1}_{\bfu}\bfu
\qquad \mbox{strongly in } \bfH^1(\D)
\end{eqnarray*}
for $n_k \rightarrow \infty$. From the equation above we know $\bfu^{n_k}$ must converge strongly in $\bfH^1(\D)$ and as we already know that the weak limit is $\bfu$, we conclude that $\bfu^{n_k}$  does in fact converge strongly to $\bfu$ and that it fulfills the equation $\bfu= \lambda\, \mathcal{L}^{-1}_{\bfu}\bfu$. Applying $\mathcal{L}_{\bfu}$ to the equation yields $E^{\prime}(\bfu)=\mathcal{L}_{\bfu}\bfu= \lambda (\bfu , \cdot )_{L^2(\D)}$, i.e., $\bfu$ is a critical point of $E$. This proves \ref{enum:3}.\\
\ref{enum:4}: If one of the limits is a locally quasi-isolated ground state, then the corresponding ground state density $|\bfu|^2$ is locally unique. In particular, there is no path of constant energy that connects $\bfu$ with another ground state (aside from phase shifts). Due to the strict energy decay of the iterates $\bfu^n$, it is hence impossible for the iterates to leave a small neighborhood of the ground state for all sufficiently large $n$, as this would require an intermediate increase of the energy. Consequently, all density iterates $|\bfu^n|^2$ must converge to the same limit. The strong convergence in $\bfL^2$ is again due to the embedding of $H^1(\D)$ into $L^4(\D)$ as used before. For a mathematical formalization of the proof we refer to the arguments elaborated in \cite{PHMY242}.
\end{proof}

\subsection{Implementation aspects and computational cost}

We briefly comment on the practical realization of the metric-driven Riemannian gradient method (MDRGM) and the associated computational cost. Recall that MDRGM, as defined in Definition \ref{definition-metric-driven-grad-method-SO-coupl-BEC}, is an abstract iterative scheme and requires a spatial discretization for its implementation.

From an algorithmic point of view, one iteration of MDRGM consists of five main steps: 
\begin{itemize}
\item[(i)] The assembly of system matrices to represent the discrete version of the elliptic differential operator $\mathcal{L}_{\bfu^n}$. Note that this requires nonlinear updates in each iteration to account for the influence of the current iterate $\bfu^n$.
\item[(ii)] The computation of $\mathcal{L}_{\bfu^n}^{-1}\bfu^n$ (i.e., the application of an inverse operator), which requires the solution of one linear system per iteration with the system matrix computed in (i) and the source term given by $\bfu^n$.
\item[(iii)] A linear update $u^{n+1}_{\pr}(\tau) := (1 - \tau)\,\bfu^n  \,+\, \tau \, ( \mathcal{L}_{\bfu^n}^{-1} \bfu^n , \bfu^n )_{\bfL^2(\D)}^{-1} \,\mathcal{L}_{\bfu^n}^{-1} \bfu^n$ which only involves matrix-vector operations once  $\mathcal{L}_{\bfu^n}^{-1}\bfu^n$ is available. For example, $( \mathcal{L}_{\bfu^n}^{-1} \bfu^n , \bfu^n )_{\bfL^2(\D)}$ can be easily assembled using the mass matrix. 
\item[(iv)] Normalization $u^{n+1}(\tau) := u^{n+1}_{\pr}(\tau) / \| u^{n+1}_{\pr}(\tau) \|_{\bfL^2(\D)}$.
\item[(v)] Computation of the optimal step size $\tau_n =
\underset{0<\tau \le 2}{\mbox{\normalfont arg\hspace{1pt}min}} \,\, E\hspace{-1pt}\left( u^{n+1}(\tau) \right)$ which can be realized at negligible cost by a line search, noting that $E\hspace{-1pt}\left( u^{n+1}(\tau) \right)$ is a rational function in $\tau$ with explicitly computable coefficients. Finally setting $u^{n+1}:=u^{n+1}(\tau^n)$.
\end{itemize}
The efficient assembly of the system matrices and the realization of the line search without relevant computational overhead are discussed in detail in \cite{PHMY242}. In typical implementations, the linear solve (ii) constitutes the dominant part of the computational cost, while the remaining steps are minor in comparison.

Hence, up to discretization-specific constants, each iteration of MDRGM is essentially equivalent to one assembly of the discrete operator and one linear solve in the chosen approximation space. The computational complexity per iteration is therefore determined by the dimension of this space together with the cost of assembly and the efficiency of the underlying linear solver. In particular, for a discretization with $N$ degrees of freedom, the cost per iteration scales like the solution of a linear system of size $N$, up to additional costs arising from the assembly of the discrete operator.

The total computational effort is thus given by the product of the number of iterations and the cost per iteration. As discussed above and observed in the numerical experiments, the number of iterations mainly depends on spectral properties of the operator and is largely independent of the discretization parameters.

In this work, MDRGM is in particular combined with metric-driven approximation spaces that admit efficient realizations via localized basis functions. This allows to keep the dimension of the discrete problem small while still capturing the characteristic features of the solution, thereby reducing the cost of each iteration without increasing the number of iterations.

\subsection{Numerical experiments}
We will now demonstrate the convergence of the algorithm to a ground state of the energy functional and we will also compare the approximation properties obtained in a standard \(\mathbb{P}^1\)-Lagrange finite element space with the approximation properties in a metric-driven approximation space in the spirit of Section \ref{subsection-LOD-space-GPE}.

To fix the setting, we consider the domain \(\D = [-1,1]^2\) and the parameters
 \[\beta_{11} = 10,\, \beta_{12}=\beta_{22} = 9,\, k_0=10,\, \Omega = 50 \,\text{   and  }\,\delta = 0.\]
 To ensure that assumption \ref{A2} is fulfilled we shift the energy level by a constant, which is achieved by selecting the trapping potentials as \(V_j(x) = \frac{\Omega + \delta + 2k_0^2}{2}\) for $j=1,2$.

 \subsubsection{Metric-driven steepest descent}
To avoid an influence from the spatial discretization and to isolate the error from the iterative solver, we fix a sufficiently accurate approximation space for the experiments in this subsection. To be precise, the exact solution space $\mathbf{H}^1_0(\D)$ is discretized using a (standard) \(\mathbb{P}^2\)-Lagrange finite element on a uniform mesh with \((2^8 - 1)^2\) degrees of freedom. Metric-driven approximation spaces are used for the experiments in the next subsection. 

In the following experiments, we compare three different iterative schemes. First, we consider the metric-driven Riemannian gradient method \eqref{RSG-SOcoupl-BEC} as stated in Definition \ref{definition-metric-driven-grad-method-SO-coupl-BEC}, abbreviated here by MDRGM.  The optimal step size \(\tau_n\) for the MDRGM is computed with the golden-section line search. Second, we consider the \quotes{inverse iteration}, abbreviated by InvIter, which is obtained by fixing $\tau=1$ in \eqref{RSG-SOcoupl-BEC}. As a third method, we also include the popular GFDN method which is obtained by a modified backward Euler discretization of an $L^2$-gradient flow with discrete normalization (GFDN), cf. 
  \cite{BC15,BaoCaiReview2018,FengTangWangIMA}. The iterations are formally given by
  \begin{eqnarray*}
\bfu^{n+1} &:=&  \frac{\hspace{-22pt}(\mbox{id} + \tau_n \mathcal{L}_{\bfu^n})^{-1} \bfu^n}{\| (\mbox{id} + \tau_n \mathcal{L}_{\bfu^n})^{-1} \bfu^n \|_{L^2(\D)}} \qquad \mbox{(GFDN)},
\end{eqnarray*}
where we select the step size uniformly as $\tau_n=1$ (since an adaptive computation of the optimal $\tau$-values is computationally not feasible for the GFDN). 
All iterative methods are initialized by interpolating the following functions in the finite element space:
\begin{align*}
	u_1^0(x_1,x_2) &= \frac{1}{2}(x_1-1)^2(x_2-1)^2\exp(-\frac{x_1^2 + x_2^2}{2}\ci)\\
	u_2^0(x_1,x_2) &= (x_1-1)^2(x_2-1)^2\exp(-\frac{x_1^2 + x_2^2}{2}\ci).
\end{align*}
  In this case, the initial energy is given by \(E({\bf u}^0) = 74.97448979636732\). The iterations were stopped as soon as the energy difference of two subsequent iterations fell below the threshold \(\varepsilon=10^{-11}\) and our calculations were carried out using \texttt{julia}. 
	\begin{figure}[H]
		\centering
		\includegraphics{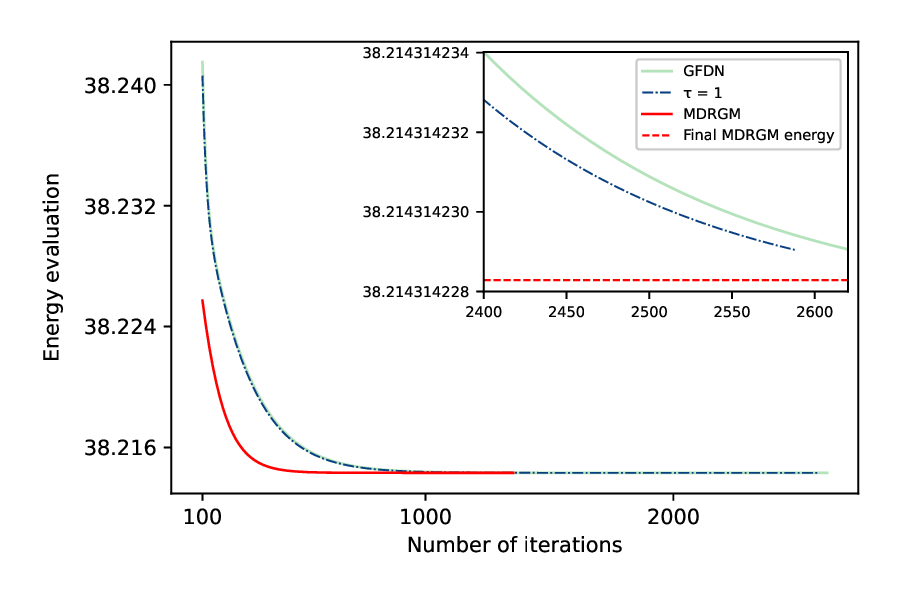}
		\caption{Comparison of the energy per iteration for the MDRGM, the inverse iteration and the GFDN.}
        \label{iteration-numbers-SO-BEC}
	\end{figure}
  In Figure \ref{iteration-numbers-SO-BEC} we can see the energy evolution with respect to the number of iterations for the three different schemes. We observe that the metric-driven Riemannian gradient method (MDRGM) shows the best performance. The inverse iteration and the GFDN took around \(2600\) iterations to reach the tolerance, while the MDRGM only required \(1354\) iterations. Here we note that the costs per time step are essentially the same for all three methods and that the computation of the optimal $\tau$ for the MDRGM does not cause any significant overhead. 
Table \ref{table-energy-final-iteration} shows the energy and eigenvalue approximation for the three methods after the final iteration.
\begin{table}[h]
	\begin{tabular}{ |p{2.5cm}||p{3cm}|p{3cm}|p{3cm}|  }
		\hline
		Method& No. of iterations & Final energy & Final eigenvalue\\
		\hline
		GFDN &2620 &38.214314229059 &78.376223785843\\
		InvIter        & 2589   &38.214314229034      &  78.376223776962\\
		MDRGM &1354    & 38.214314228286  &78.376223454563\\
		\hline
	\end{tabular}%
	\caption {Comparison of final results}
    \label{table-energy-final-iteration}
\end{table}
Evidently, the MDRGM is not only faster, but also yields a slightly smaller energy compared to the inverse iteration, which in turn also outputs a slightly smaller final energy compared to the GFDN.

To verify whether a found state $\bfu$ is a (local) minimizer, we recall from Section \ref{subsection:first-second-order-cond} that we have to check the first and the second order condition for minimizers. For that we compute the residuals in the first order condition $E^{\prime}(\bfu)=\lambda (\bfu,\cdot)_{\bfL^2(\D)}$ and find that they range between $O(10^{-8})$ and $O(10^{-10})$ in the maximum norm for the various methods. Hence, the first order condition is fulfilled up to the numerical precision of the methods. For the second order condition, we have to verify that $\lambda$ is the unique smallest eigenvalue of the operator \(E''(\bfu)|_{T_{\bfu}\boldS}\). The results are depicted in Table \ref{table-eigenvalues-secE}. 
We can confirm that $\lambda$ is indeed simple and at the bottom of the spectrum. Furthermore, we observe that the smallest eigenvalues of \(E''(\bfu)|_{T_{\bfu}\boldS}\) lie very close together, which often indicates the presence of interesting physical phenomena, as demonstrated by the density plots.
\begin{table}[h] 
	\begin{tabular}{ |p{2.5cm}||p{3cm}|p{3cm}|p{3cm}|  }
		\hline
		$i$& 1 & 2& 3\\
		\hline
		GFDN &78.376223779431 &78.636660348051 &78.651177504685\\
		InvIter     & 78.376223770582   &78.636660410597       &  78.651177520292 \\
		MDRGM &78.376223448183    &78.636662675963 &78.651178085638\\
		\hline
	\end{tabular}
	\caption{The three smallest eigenvalues of \(E''(\bfu)|_{T_{\bfu}\boldS}\) sorted in ascending order. Here, $\bfu$ denotes the state after the final iteration obtained with GFDN, inverse iteration and MDRGM respectively.}
    \label{table-eigenvalues-secE}
\end{table}
\begin{figure}[ht!]
	\centering
	    \centering
    \includegraphics[width=0.45\linewidth]{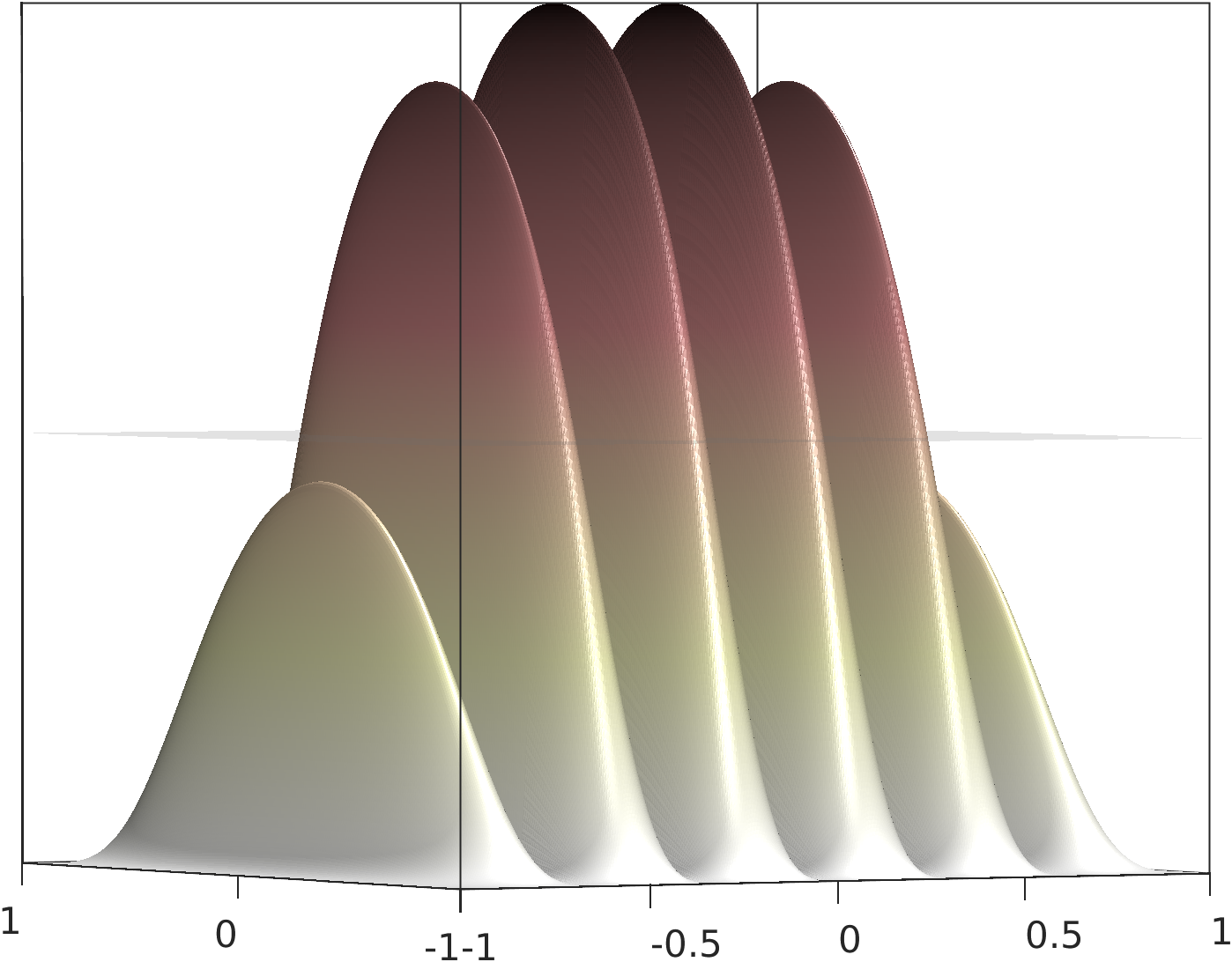}
    \includegraphics[width=0.45\linewidth]{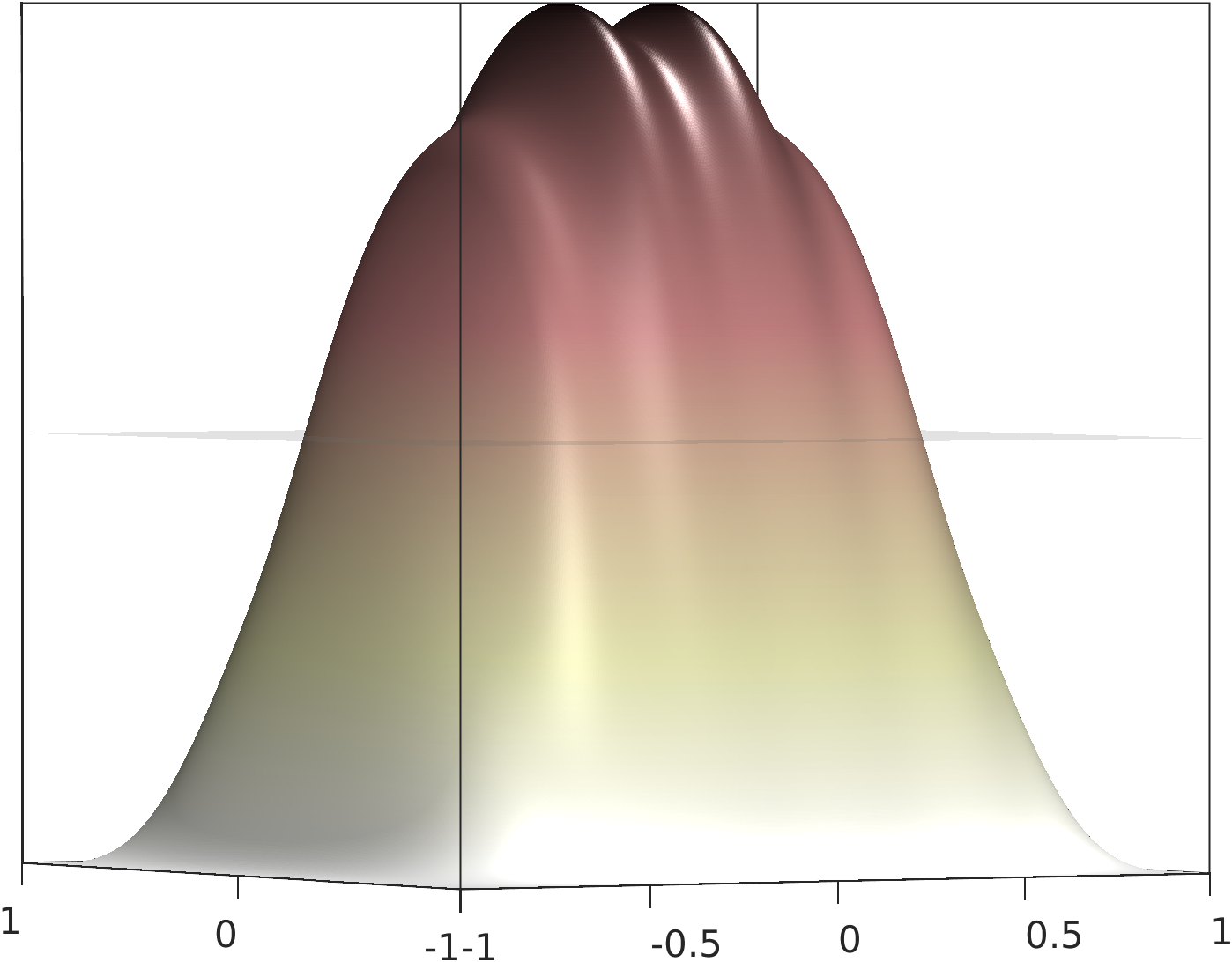}
    \caption{Converged ground state density $|\bfu|^2$. The left picture shows the density of the first component and the right picture the density of the second component.}
    \label{denisty-plots-so-coupled-BEC}
\end{figure}

The density plots in Figure \ref{denisty-plots-so-coupled-BEC} reveal the formation of disk-like structures in the first component, characterized by alternating regions of high and low density. This pattern is indicative of a supersolid phase, meaning that the state of matter exhibits properties of solids and superfluids at the same time.

\subsubsection{Metric-driven approximation spaces}
We next investigate the convergence of the metric-driven approximation spaces (LOD) towards the fully resolved solution of \eqref{LOD-minimizer-GPE}. To this end, we consider the spaces $V_H^{\mathcal{L}_0}$ defined in \eqref{LOD-space-GPE} with respect to a sequence of coarse triangulations with mesh widths $H = 2^{-1}, \ldots, 2^{-5}$. Recall that this construction of the metric-driven space neglects nonlinear contributions and does not require updates.

In order to make the method computationally feasible, we employ two simplifications. First, we localize the correctors, i.e., we compute localized basis functions on finite element patches of width $|\log_2(H)|$, which ensures that the sparsity pattern is preserved. This follows standard practice in LOD theory \cite{MaP14,HeP13,LODbook21,ActaLOD21}. Second, we omit the coupling term in the construction of the spaces. As a consequence, the approximation spaces decouple, and the same precomputed basis can be used for both components.  

Even with these simplifications, which improve the efficiency of the method, the convergence rates predicted by Theorem~\ref{theorem:LOD-estimates-GPE} for the simplified setting are still observed in the present, more challenging case of a spin-orbit coupled BEC; see Figure~\ref{fig:FEM-conv}. Already at refinement level~4, the practical metric-driven space captures the characteristic disk-like structure of the ground state; see Figure~\ref{fig:LOD-lvl4}.  

For comparison, Figure~\ref{fig:FEM-conv} also reports the errors obtained with standard $\mathbb{P}^1$ finite elements on the same meshes. As expected, this problem-agnostic method converges, but at a significantly slower rate when measured against the number of degrees of freedom. The logarithmic communication overhead inherent in the metric-driven method improves the approximation quality substantially. 

Given that the iterative solvers for this setup require order $10^3$ iterations, the additional cost of precomputing the metric-driven basis is justified. Overall, these results provide a promising proof of concept for the LOD approach in this context. A more systematic and rigorous investigation will be the subject of future research.  

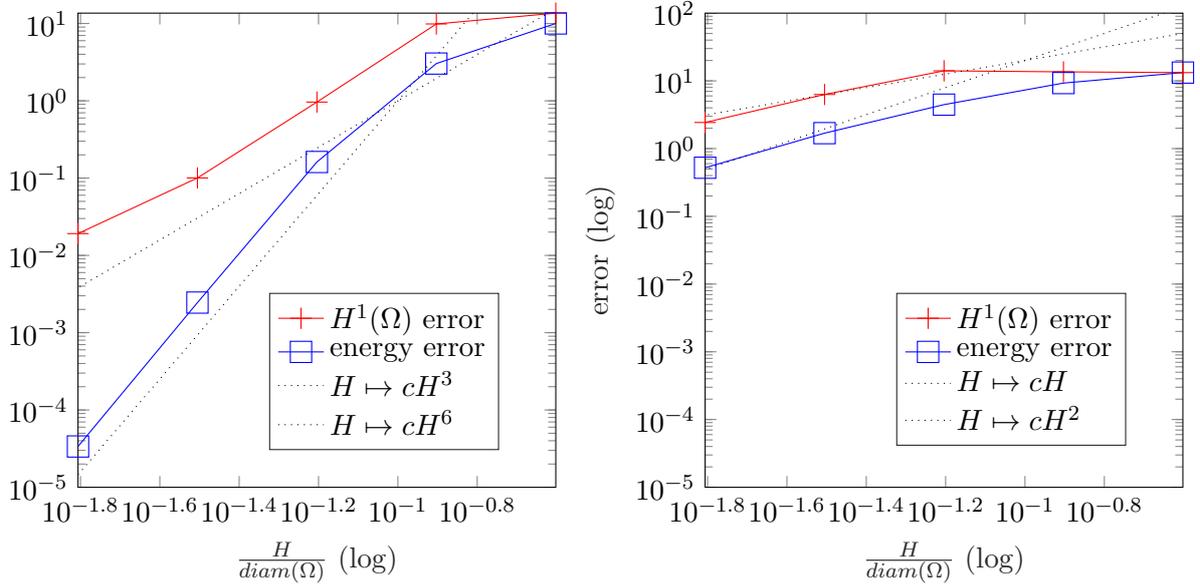
\begin{figure}
    \centering
%
%
\definecolor{mycolor1}{rgb}{1.00000,0.00000,1.00000}%
\begin{tikzpicture}

\begin{axis}[%
width=2.475in,
height=2.475in,
scale only axis,
xmode=log,
xmin=0.015625,
xmax=0.25,
xminorticks=true,
xlabel style={font=\color{white!15!black}},
xlabel={$\frac{H}{diam(\Omega)}$ (log)},
ymode=log,
ymin=1e-5,
ymax=13.633738265848,
yminorticks=true,
ylabel style={font=\color{white!15!black}},
axis background/.style={fill=white},
legend style={at={(0.4,0.41)}, anchor=north west, legend cell align=left, align=left, draw=white!15!black}
]
\addplot [color=red, mark size=4.0pt, mark=+, mark options={solid, red}]
  table[row sep=crcr]{%
0.25	13.633738265848\\
0.125	9.87591105486042\\
0.0625	0.963434440138719\\
0.03125	0.100470857078415\\
0.015625	0.0191306895066592\\
};
\addlegendentry{$H^1(\Omega)$ error}


\addplot [color=blue, mark size=4.0pt, mark=square, mark options={solid, blue}]
  table[row sep=crcr]{%
0.25	10.0243902163155\\
0.125	3.03424963107367\\
0.0625	0.16166159996542\\
0.03125	0.00245186287487087\\
0.015625	3.37175000879597e-05\\
};
\addlegendentry{energy error}

\addplot [color=black, dotted]
  table[row sep=crcr]{%
0.25	15.625\\
0.125	1.953125\\
0.0625	0.244140625\\
0.03125	3.0517578125e-02\\
0.015625	3.814697265625e-03\\
};
\addlegendentry{$H\mapsto cH^3$}


\addplot [color=black, dotted]
  table[row sep=crcr]{%
0.25	244.140625\\
0.125	3.814697265625e-00\\
0.0625	5.96046447753906e-02\\
0.03125	9.31322574615479e-04\\
0.015625	1.45519152283669e-05\\
};
\addlegendentry{$H\mapsto cH^6$}

\end{axis}
\end{tikzpicture}%
\vspace{1ex}
%
%
\definecolor{mycolor1}{rgb}{1.00000,0.00000,1.00000}%
\begin{tikzpicture}

\begin{axis}[%
width=2.475in,
height=2.475in,
scale only axis,
xmode=log,
xmin=0.015625,
xmax=0.25,
xminorticks=true,
xlabel style={font=\color{white!15!black}},
xlabel={$\frac{H}{diam(\Omega)}$ (log)},
ymode=log,
ymin=1e-5,
ymax=100,
yminorticks=true,
ylabel style={font=\color{white!15!black}},
ylabel={error (log)},
axis background/.style={fill=white},
legend style={at={(0.4,0.41)}, anchor=north west, legend cell align=left, align=left, draw=white!15!black}
]
\addplot [color=red, mark size=4.0pt, mark=+, mark options={solid, red}]
  table[row sep=crcr]{%
0.25	13.250153870372\\
0.125	13.5873776673012\\
0.0625	14.0756518339241\\
0.03125	6.2952532976866\\
0.015625	2.43063182199618\\
};
\addlegendentry{$H^1(\Omega)$ error}


\addplot [color=blue, mark size=4.0pt, mark=square, mark options={solid, blue}]
  table[row sep=crcr]{%
0.25	13.2602631517299\\
0.125	9.27839414163681\\
0.0625	4.45810381317148\\
0.03125	1.6908173954983\\
0.015625	0.522081011027225\\
};
\addlegendentry{energy error}

\addplot [color=black, dotted]
  table[row sep=crcr]{%
0.25	50\\
0.125	25\\
0.0625	12.5\\
0.03125	6.25\\
0.015625 3.125\\
};
\addlegendentry{$H\mapsto cH$}

\addplot [color=black, dotted]
  table[row sep=crcr]{%
0.25	125\\
0.125	31.25\\
0.0625	7.8125\\
0.03125	1.953125\\
0.015625	0.48828125\\
};
\addlegendentry{$H\mapsto cH^2$}

\end{axis}
\end{tikzpicture}%
    \caption{Convergence of the metric-driven approximation (LOD, left) compared with standard $\mathbb{P}^1$-FEM (right). Shown are the $H^1$ and energy errors with respect to a sequence of uniformly refined triangular meshes of width $H$.}
    \label{fig:FEM-conv}
\end{figure}

\begin{figure}
    \centering
    \includegraphics[width=0.45\linewidth]{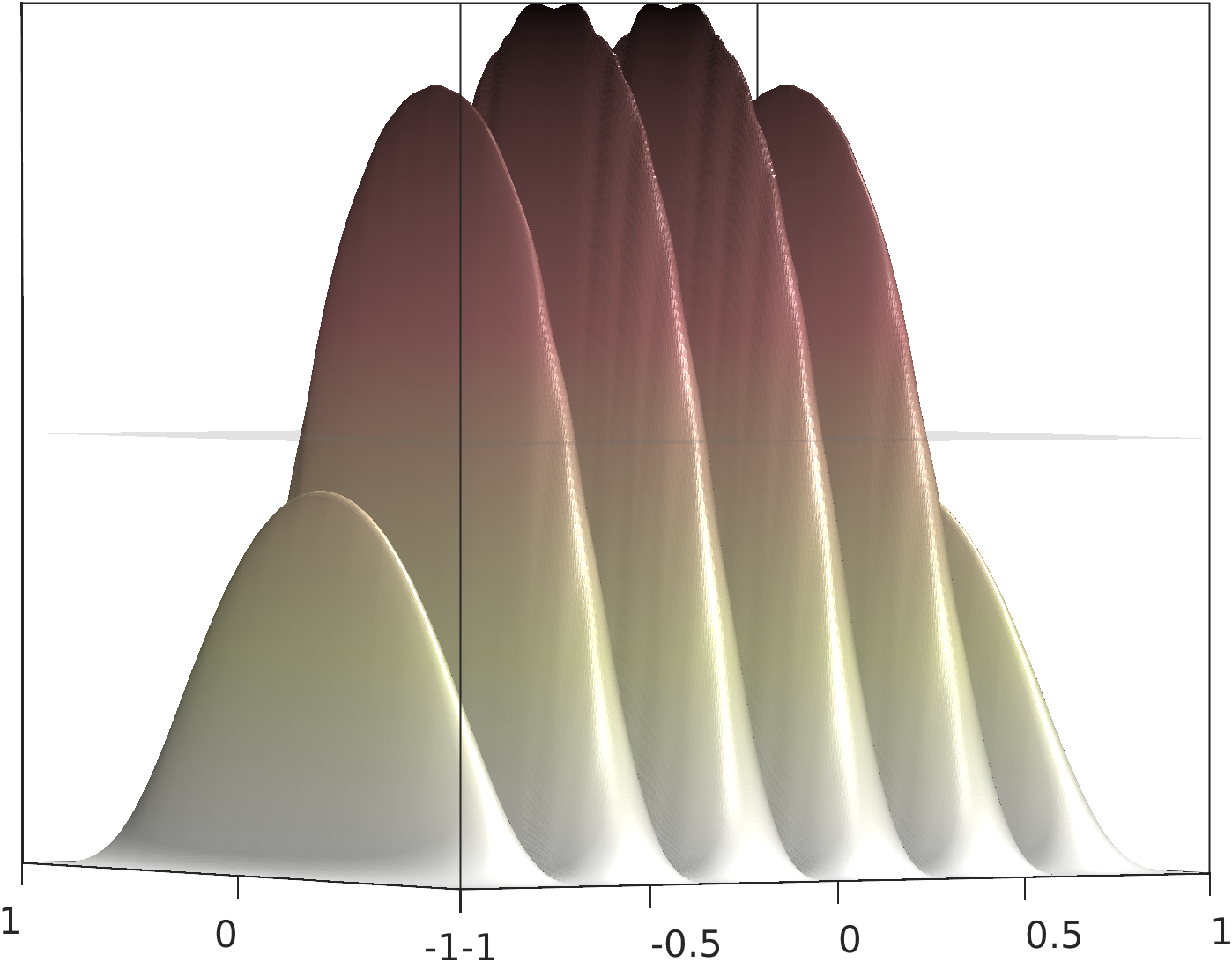}
    \includegraphics[width=0.45\linewidth]{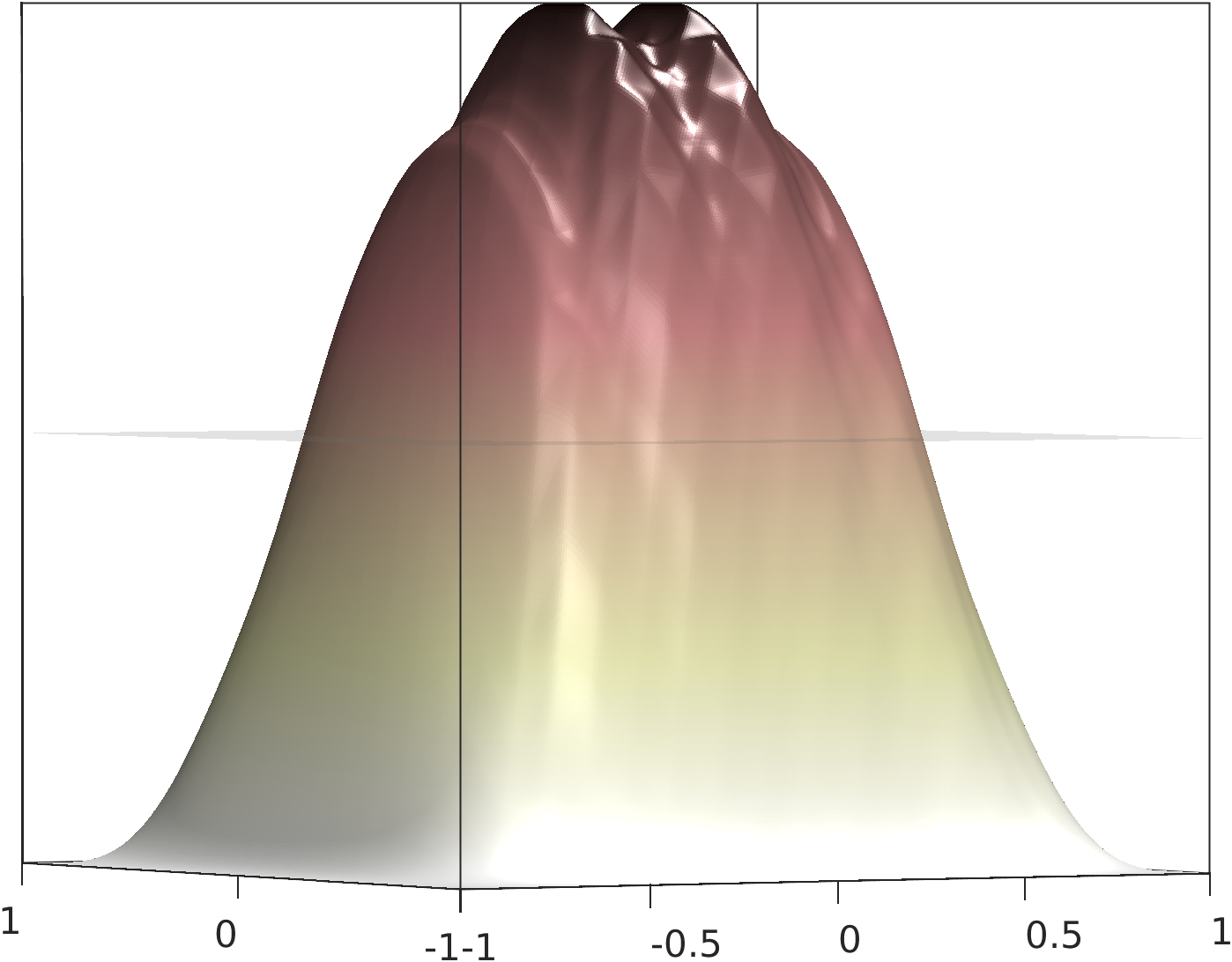}
    \caption{Ground state density $|\bfu|^2$ obtained with the LOD method on refinement level~4. Left: first component. Right: second component. The main qualitative features of the disk-like structure are already captured at this resolution.}
    \label{fig:LOD-lvl4}
\end{figure}

\bibliographystyle{abbrv} 
\bibliography{references.bib}

\end{document}